\documentclass[10pt]{article}

\usepackage{hyperref}
\usepackage{amsfonts,amssymb,amsmath,amscd,euscript,
array,mathrsfs}
\input{liemacs10.sty} 
\newcommand{\rlg}{\breve\delta}
\newcommand{\oo}{{\overline{1}}}

\newcommand{\ood}{{\overline{1}}}
\newcommand{\eev}{{\overline{0}}}

\newcommand{\sseq}{\subseteq}

\renewcommand{\oline}{\overline}

\renewcommand{\subeq}{\subseteq} 
\renewcommand{\la}{\langle}
\renewcommand{\ra}{\rangle}

\newcommand\Evol{\mathop{\rm Evol}\nolimits}

\renewcommand{\mlabel}{\label}

\begin{document} 


\title{Differentiable vectors and unitary representations of 
Fr\'echet--Lie supergroups}

\author{Karl--Hermann Neeb
\begin{footnote}{
Department Mathematical,
FAU Erlangen-N\"urnberg,
Cauerstra\ss e 11, 91058 Erlangen, Deutschland,
\texttt{karl-hermann.neeb@math.uni-erlangen.de}
}
\end{footnote}
\and Hadi Salmasian\begin{footnote}
{
Department of Mathematics and Statistics,
University of Ottawa, 585 King Edward Ave., Ottawa, ON K1N 6N5,
Canada,
\texttt{hsalmasi@uottawa.ca}
}
\end{footnote}
\begin{footnote}{
The second author was supported by an NSERC Discovery Grant and  the Emerging Field Program at Universit\" at Erlangen--N\"urnberg.
}\end{footnote}
}

\maketitle

\begin{abstract}  A locally convex 
Lie group $G$ has the Trotter property if, for every $x_1, x_2 \in \g$, 
\[ \exp_G(t(x_1 + x_2)) = \lim_{n \to \infty} 
\Big(\exp_G\Big(\frac{t}{n}x_1\Big)\exp_G\Big(\frac{t}{n}x_2\Big)\Big)^n \] 
holds uniformly on compact subsets of $\R$. All locally exponential Lie groups have this property, but also 
groups of automorphisms of principal bundles over compact smooth manifolds. A key result 
of the present article is that, if $G$ has the Trotter property, $\pi \: G \to \GL(V)$ is a continuous representation 
of $G$ on a locally convex space, 
and  $v \in V$  is a vector such that
$\oline{\dd\pi}(x)v := \frac{d}{dt}|_{t=0} \pi(\exp _G(tx))v$ exists for every 
$x \in \g$, then the map $\g \to V, x \mapsto \oline{\dd\pi}(x)v$ is linear. 

Using this result we conclude that, 
for a representation of a locally exponential Fr\'echet--Lie group 
$G$ on a metrizable locally convex space, 
the space of $\cC^k$-vectors coincides with the 
common domain of the $k$-fold products of the operators 
$\oline{\dd\pi}(x)$. For unitary representations on Hilbert spaces, the assumption 
of local exponentiality can be weakened to the Trotter property. 

As an application, we show that for smooth (resp., analytic) unitary representations of Fr\'echet--Lie supergroups $(G,\g)$ where $G$ has 
the Trotter property, the common domain of the operators of $\g = \g_\eev \oplus \g_\ood$ 
can always be extended to the space of smooth (resp., analytic) vectors for~$G$. 
\end{abstract}


\section*{Introduction}

Let $G$ be a Lie group modeled on a 
locally convex space. Assume that
$G$ has a (smooth) exponential function
$\exp_G \: \L(G) = \g \to G$, so that every smooth one-parameter group 
of $G$ is of the form 
$\gamma_x(t): = \exp_G(tx)$ for some $x \in \L(G)$. For every 
continuous representation $\pi \: G \to \GL(V)$ of $G$ on a locally convex space 
$V$, set
$$\cD_x := \Big\{ v \in V \: \derat0 \pi(\exp_G (tx))v \ \mbox{ exists } \Big\}.$$ 
Thus, $\cD_x$ is the domain 
of the infinitesimal generator 
$$
\oline{\dd\pi}(x)v := \derat0 \pi(\exp_G(tx))v
$$ 
of the one-parameter group $\pi(\exp_G(tx))$. 
Let
 $\cD^1 := \bigcap_{x \in \g} \cD_x$ be the common
 domain  of these operators. 
Clearly, each vector $v$ with a continuously differentiable orbit map (a $\cC^1$-vector) 
is contained in $\cD^1$, but the converse problem turns out to be a tricky question. 
The main difficulty is to establish that, for $v \in \cD^1$, 
the map 
\[ \omega_v \: \g \to V, \quad x \mapsto \oline{\dd\pi}(x)v \] 
is linear and, if this is the case, and $\omega_v$ is continuous, 
to show that this implies that $v$ is a $\cC^1$-vector. As we have seen in 
\cite{Ne10}, the latter problem can be solved rather easily if $G$ is locally 
exponential, but in practice this assumption appears rather strong because 
it is not satisfied for groups of diffeomorphisms. 
With respect to the linearity of $\omega_v$, 
we managed to show in \cite[Thm.~2.8]{Ne10} that 
$\omega_v$ is linear if $G$ is a Banach--Lie group and $V$ is a Banach space. 
To achieve this result, we had to build on quite involved 
results of Neklyudov (\cite{Nek08}). 

In the present note we obtain much more satisfactory solutions to both 
problems with rather direct proofs. The key to our new approach is the 
recent paper \cite{BB11} by I.~and D.~Belti\c{t}\u{a}, where 
they address the linearity problem for $\omega_v$ in the context of 
topological groups. In Section~\ref{sec:s2}, we show that their approach 
can actually be carried much further to obtain our 
Theorem~\ref{thm:3.4}, which asserts that for a Lie group 
$G$ with the Trotter property, and a continuous function 
$\phi:G\to V$ (where $V$ is a locally convex space),
if 
the derivatives $D_x \phi(g) := \frac{d}{dt}|_{t = 0} 
 \phi(g\gamma_x(t))$ with respect to all 
 one-parameter 
groups $\gamma_x$, $x \in \g$, are continuous maps, 
then the map $x \mapsto D_x\phi$ is linear. 

That a Lie group $G$ has the {\it Trotter property} means that, for $x_1, x_2 \in \g$, 
\[ \exp_G(t(x_1 + x_2)) = \lim_{n \to \infty} 
\Big(\exp_G\Big(\frac{t}{n}x_1\Big)\exp_G\Big(\frac{t}{n}x_2\Big)\Big)^n \] 
holds uniformly on compact subsets of $\R$. 
The main advantage of this property 
is that the class of Lie groups with this property contains 
all locally exponential Lie groups (hence all Banach--Lie groups), 
groups of automorphisms of principal bundles over
compact smooth manifolds (in particular, diffeomorphism groups of  compact smooth manifolds), 
and direct limit Lie groups 
(see Section \ref{sec:examples}). 

Theorem~\ref{thm:3.4} turns out to be just the right tool to deal with differentiable 
vectors in continuous representations of Lie groups. 
Combining it with the techniques 
developed in \cite{Ne10}, we show in Section~\ref{sec:3} that, for a continuous representation 
of a locally exponential Fr\'echet--Lie group on a metrizable locally convex space,
$\cD^1$ coincides with the space of $\cC^1$-vectors, and there is a similar 
characterization of $\cC^k$-vectors (Theorem~\ref{thm:2.9}). 
This already generalizes the corresponding Banach results from \cite{Ne10} considerably 
with substantially simpler proofs.  

In Section~\ref{sec:4} we turn to the special class of 
unitary representations which are differentiable 
in the sense that the space of $\cC^1$-vectors is dense. 
For these we can weaken the assumption of  
local exponentiality to the Trotter property. Namely, for any Fr\'echet--Lie group $G$ 
with the Trotter property,
we obtain
the natural characterization of the space 
of $\cC^k$-vectors as the common domain of the $k$-fold products of the operators 
$\oline{\dd\pi}(x)$, $x \in \g$ (Theorem~\ref{thm:2.1}). 

In Section~\ref{sec:5} we apply all this to unitary representations 
of Lie supergroups $(G,\g)$, which we consider as a pair consisting of a 
Lie superalgebra $\g=\g_\eev\oplus\g_\ood$ and a Lie group $G$ whose Lie algebra is the even part 
$\g_\eev$ of $\g$ (see \cite{CCTV06}). A crucial difficulty in dealing with unitary representations 
of Lie supergroups is the specification of the common domain 
of the operators corresponding to the odd part $\g_\ood$ (see \cite{CCTV06} and \cite{MNS11} for a detailed discussion). 
For the large class of Fr\'echet--Lie supergroups where $G$ has the Trotter property, 
we roughly show that, if the representation of $G$ is smooth, resp., analytic, 
the common domain of the operators of $\g = \g_\eev \oplus \g_\ood$ 
can always be extended to the space of smooth, resp., analytic vectors for~$G$. 
This generalizes the respective stability results for Banach--Lie supergroups from \cite{MNS11} and 
for finite-dimensional Lie supergroups from \cite{CCTV06}.  
We thus obtain a natural context for a global unitary representation theory 
for Lie supergroups modeled on Fr\'echet spaces, which applies in particular 
to diffeomorphism groups and gauge groups.

\tableofcontents

\section{Locally convex Lie groups} \mlabel{sec:1}

In this section we briefly recall the basic concepts related 
to infinite-dimensional Lie groups. Throughout these notes all topological 
groups and vector spaces are assumed to be Hausdorff.

\begin{defn} (a) 
Let $E$ and $F$ be locally convex spaces, $U
\subeq E$ open and $f \: U \to F$ a map. Then the {\it derivative
  of $f$ at $x$ in the direction $h$} is defined as 
$$ \dd f(x)(h) := (\partial_h f)(x) := \derat0 f(x + t h) 
= \lim_{t \to 0} \frac{1}{t}(f(x+th) -f(x)) $$
whenever it exists. The function $f$ is called {\it differentiable at
  $x$} if $\dd f(x)(h)$ exists for all $h \in E$. It is called {\it
  continuously differentiable}, if it is differentiable at all
points of $U$ and 
$$ \dd f \: U \times E \to F, \quad (x,h) \mapsto \dd f(x)(h) $$
is a continuous map. Note that this implies that the maps 
$\dd f(x)$ are linear (cf.\ \cite[Lemma~2.2.14]{GN12}). 
The map $f$ is called a {\it $\cC^k$-map}, $k \in \N \cup \{\infty\}$, 
if it is continuous, the iterated directional derivatives 
$$ \dd^{j}f(x)(h_1,\ldots, h_j)
:= (\partial_{h_j} \cdots \partial_{h_1}f)(x) $$
exist for all integers $1\leq j \leq k$, $x \in U$ and $h_1,\ldots, h_j \in E$, 
and all maps $\dd^j f \: U \times E^j \to F$ are continuous. 
As usual, $\cC^\infty$-maps are called {\it smooth}. 

  (b) If $E$ and $F$ are complex locally convex spaces, then $f$ is 
called {\it complex analytic} if it is continuous and for each 
$x \in U$ there exists a $0$-neighborhood $V$ with $x + V \subeq U$ and 
continuous homogeneous polynomials $\beta_k \: E \to F$ of degree $k$ 
such that for each $h \in V$ we have 
$$ f(x+h) = \sum_{k = 0}^\infty \beta_k(h), $$
as a pointwise limit (\cite{BS71}). 
The map $f$ is called {\it holomorphic} if it is $\cC^1$ 
and for each $x \in U$ the 
map $\dd f(x) \: E \to F$ is complex linear (cf.\ \cite[p.~1027]{Mil84}). 
If $F$ is sequentially complete, then $f$ is holomorphic if and only if 
it is complex analytic (\cite[Ths.~3.1, 6.4]{BS71}). 

(c) If $E$ and $F$ are real locally convex spaces, 
then we call a map $f \: U \to F$, $U \subeq E$ open, 
{\it real analytic} or a $\cC^\omega$-map, 
if for each point $x \in U$ there exists an open neighborhood 
$V \subeq E_\C$ and a holomorphic map $f_\C \: V \to F_\C$ with 
$f_\C\res_{U \cap V} = f\res_{U \cap V}$  (cf.\ \cite{Mil84}). 
The advantage of this definition, which differs from the one in 
\cite{BS71}, is that it also works nicely for non-complete spaces. 
Any analytic map is smooth, 
and the corresponding chain rule holds without any condition 
on the underlying spaces, which is the key to the definition of 
analytic manifolds (see \cite{Gl02} for details).
\end{defn}

Once the concept of a smooth function 
between open subsets of locally convex spaces is established  
(cf.\ \cite{Ne06}, \cite{GN12}), it is clear how to define 
a locally convex smooth manifold. 
A {\it (locally convex) Lie group} $G$ is a group equipped with a 
smooth manifold structure modeled on a locally convex space 
for which the group multiplication and the 
inversion are smooth maps. We write $\1 \in G$ for the identity element. 
Then each $x \in T_\1(G)$ corresponds to
a unique left invariant vector field $x_l$ with 
$x_l(\1) = x$. The space of left invariant vector fields is closed under the Lie
bracket of vector fields, hence inherits a Lie algebra structure. 
In this sense we obtain on $\g := T_\1(G)$ a continuous Lie bracket which
is uniquely determined by $[x,y] = [x_l, y_l](\1)$ for $x,y \in \g$. 
We shall also use the functorial notation $\L(G) := (\g,[\cdot,\cdot])$ 
for the Lie algebra of $G$. 
The adjoint action of $G$ on $\g$ is defined by 
$\Ad(g) := T_\1(c_g)$, where $c_g(x) = gxg^{-1}$ is the conjugation map. 
The adjoint action is smooth and 
each $\Ad(g)$ is a topological isomorphism of $\g$. 
If $\g$ is a Fr\'echet, resp., a Banach space, then 
$G$ is called a {\it Fr\'echet-}, resp., a 
{\it Banach--Lie group}. 

A smooth map $\exp_G \: \g \to G$  is called an {\it exponential function} 
if each curve $\gamma_x(t) := \exp_G(tx)$ is a one-parameter group 
with $\gamma_x'(0)= x$. The Lie group $G$ is said to be 
{\it locally exponential} 
if it has an exponential function for which there is an open $0$-neighborhood 
$U$ in $\g$ mapped diffeomorphically by $\exp_G$ onto an 
open subset of $G$. If, in addition, $G$ is analytic and 
the exponential function is an analytic 
diffeomorphism in a $0$-neighborhood, 
then $G$ is called a {\it BCH--Lie group} (for Baker--Campbell--Hausdorff). 
The class of BCH--Lie groups contains in particular all Banach--Lie groups 
(\cite[Prop.~IV.1.2]{Ne06}).

\section{Linearity of differentials} \mlabel{sec:2}
\label{sec:s2}

The main result of this section is Theorem~\ref{thm:3.4}. 
It extends \cite[Thm.~2.5]{BB11} in the sense that it weakens the necessary assumptions 
considerably and thus makes it much simpler to apply. Its main application 
is that, for a continuous representation 
of the Lie group $G$ (with the Trotter property) on a locally convex space $V$
and an element $v \in V$ for which 
$\omega_v(x) := \frac{d}{dt}|_{t=0} \pi(\exp_G (tx))v$ exists for every $x \in \L(G)$, 
the map $\omega_v \: \g \to V$ 
is linear (Theorem~\ref{thm:additive}). For continuous representations 
of Banach--Lie groups on Banach spaces this was known 
from \cite[Thm.~8.2]{Ne10}, but the proof used quite involved 
results of Neklyudov (\cite{Nek08}) which are now bypassed. Therefore 
our results also provide a considerable simplification of the arguments for 
Banach--Lie groups.

Let $S\subseteq\R$ and $G$ be a topological group. We say that a sequence of functions $\gamma_n:S\to G$ is \emph{uniformly convergent} to $\gamma:S\to G$ if for every $\1$-neighborhood $U\subseteq G$ there exists an $N\geq 1$ such that $\gamma_n(s)\in\gamma(s)U$ 
for every $s\in S$ and every $n\geq N$. 
If the maps $\gamma_n:S\to G$ are continuous then $\gamma$ is continuous as well. If, in addition, $S$ is compact, then the definition is symmetric, that is, 
for every $\1$-neighborhood $U\subseteq G$ there exists an $N\geq 1$ such that
 $\gamma_n(s)\in U\gamma(s)$ for  every $s\in S$ and every $n\geq N$.

\begin{defn} (cf.\ \cite{BB11}) 
(a) Let $G$ be a topological group and 
$\fL(G) := \Hom(\R,G)$ denote the set of continuous 
one-parameter groups. We endow this set with the 
topology of uniform convergence on compact subsets of~$\R$. 
For $X, X_1, X_2 \in \fL(G)$ we say that $X = X_1 + X_2$ if 
\begin{equation}
  \label{eq:trotter} 
X(t) = \lim_{n \to \infty} 
\Big(X_1\Big(\frac{t}{n}\Big)X_2\Big(\frac{t}{n}\Big)\Big)^n 
\end{equation}
holds uniformly in $t$ on compact subsets of $\R$. 

(b) We say that a Lie group $G$ 
has the {\it Trotter property} if 
it has a smooth exponential function and the one-parameter groups 
$\gamma_x(t) := \exp_G(tx)$, $x \in \L(G)$, satisfy 
$\gamma_{x_1 + x_2} = \gamma_{x_1} + \gamma_{x_2}$ 
in the sense defined above, i.e., 
\[ \exp_G(t(x_1 + x_2)) = \lim_{n \to \infty} 
\Big(\exp_G\Big(\frac{t}{n}x_1\Big)\exp_G\Big(\frac{t}{n}x_2\Big)\Big)^n \text{ \ for all }x_1,x_2\in\L(G),\] 
uniformly on compact subsets of $\R$. 

\end{defn}

\begin{lem} \mlabel{lem:2.2} Let $G$ be a topological group and 
$\gamma_n \: I = [0,1] \to G$ be a sequence of continuous curves 
converging uniformly to a limit curve 
$\gamma$. Then 
\[ \gamma(I) \cup \bigcup_{n = 1}^\infty \gamma_n(I) \]
is a compact subset of $G$. 
\end{lem}

\begin{prf} Let 
$X := Y \times I$, where 
$Y := \{0\} \cup \{ \frac{1}{n} \: n \in \N\} \subeq \R$. 
Then $Y$ is compact, and therefore $X$ is a compact subset of 
$\R^2$. The map 
$F \: X \to G$ defined by 
\[ F(s,t) := 
\begin{cases} 
\gamma_n(t) & \text{ for } s = \frac{1}{n} \\ 
\gamma(t) & \text{ for } s = 0 
\end{cases} \] 
is continuous, and therefore its image is a compact subset of~$G$.
\end{prf}

\begin{lem} \mlabel{lem:2.3} Let $G$ and $H$ be topological groups, $K \subeq G$ be 
a compact subset and 
$f \: G \to H$ continuous. Then $f$ is uniformly continuous on $K$ in the following sense: 
For every $\1$-neighborhood $U_H$ in $H$ there exists a $\1$-neighborhood $U_G$ in 
$G$ such that 
\[ (\forall x \in K)\ f(xU_G) \subeq f(x) U_H.\] 
\end{lem}

\begin{prf} This essentially follows from Remark \ref{rem:3.9}(b) but we also give a direct proof. The map 
\[ F \: K \times G \to H, \quad F(x,z) := f(x)^{-1}f(xz)\] 
is continuous, so that the inverse image 
$F^{-1}(U_H^0)$ of the interior $U_H^0$ of $U_H$ 
is an open subset of $K \times G$ containing 
$K \times \{\1\}$. Hence there exists an open subset $U_G \subeq G$ 
with $K \times U_G \subeq F^{-1}(U_H^0)$. 
\end{prf}

\begin{defn} (\cite[Def.~2.1]{BB11})
Let $G$ be a topological group and $V$ be a locally convex space.  
For $\phi\colon G\to V$, $X\in\fL(G)$, and $g\in G$,  we write 
\begin{equation}\label{aux4_eq1}
(D_X\phi)(g)=\lim_{t\to 0}\frac{\phi(gX(t))-\phi(g)}{t} 
\end{equation}
whenever the limit on the right-hand side exists. 
\end{defn}

The next lemma is a 
more accurate version of
\cite[Prop. 2.2]{BB11}. We denote the space of continuous maps from $G$ into $V$ by $\cC(G,V)$.

\begin{lem} \mlabel{lem:2.5} 
Let $G$ be a topological group, $V$ be a locally convex space 
and $\phi \in \cC(G,V)$. Then the following assertions hold: 
\begin{description} 
\item[\rm(i)] Let $X \in \fL(G)$ and $n \geq 1$ be such that the function 
$D_X^n \phi$ is defined and continuous. 
For every $g\in G$ and $t\in\R$ we then have 
\[ \phi(gX(t))
=\sum_{j=0}^n\frac{t^j}{j!}(D_X^j\phi)(g)+t^n\chi_1^X(g,t)\] 
where $\chi_1^X\colon G\times\R\to V$ is continuous 
and $\chi_1^X(g,0) = 0$ for every $g\in G$. 
\item[\rm(ii)] If $X_1,X_2\in\fL(G)$ are such that 
$D_{X_1}\phi$ and $D_{X_2}\phi$ are defined and continuous, then 
\[ \phi(gX_1(t)X_2(t))
=\phi(g)+t((D_{X_1}+D_{X_2})\phi)(g)+t\chi_2(g,t)\quad \mbox{ for all } \quad (g,t) \in G \times \R,\] 
 where $\chi_2\colon G\times\R\to V$  
is a continuous function satisfying 
$\chi_2(g,0) = 0$ for every $g\in G$. 
\end{description}
\end{lem} 

\begin{prf} (i) The Taylor formula leads to the asserted identity with 
the remainder term 
\[ \chi_1^X(g,t) = 
\frac{1}{(n-1)!}\int_0^1\Bigl((1-s)^{n-1}(D_X^n\phi)(gX(ts))
-\frac{1}{n}(D_X^n\phi)(g)\Bigr)\, ds 
\] 
which is continuous on $G \times \R$ because the function $D_X^n\phi$ 
is continuous and hence one can use Lemma \ref{lem:2.3} for compact sets of the form $\{g\}\times [0,1]\subseteq G\times\R$. 

(ii) First we use (i) to obtain a continuous function 
$\chi_1^{X_2}$ on $G \times \R$ vanishing in $(g,0)$ and satisfying 
\[ \phi(gX_2(t))=\phi(g)+t(D_{X_2}\phi)(g)+t\chi_1^{X_2}(g,t) \quad \mbox{ for } \quad 
(g,t) \in G \times \R.\] 
This leads immediately to 
\[ \phi(gX_1(t)X_2(t))=\phi(gX_1(t))+t(D_{X_2}\phi)(gX_1(t))
+t\chi_1^{X_2}(gX_1(t),t).\] 
From (i) we also obtain a continuous function 
$\chi_1^{X_1}$ on $G \times \R$ vanishing in $(g,0)$ and satisfying 
\[ \phi(gX_1(t))=\phi(g)+t(D_{X_1}\phi)(g)+t\chi_1^{X_1}(g,t)
\quad \mbox{ for } \quad (g,t) \in G \times \R.\] 
By plugging this formula into the previous one, we get 
\[ \phi(gX_1(t)X_2(t))
=\phi(g)+t(D_{X_1}\phi)(g)+t(D_{X_2}\phi)(g) +t\chi_2(g,t),\] 
where 
\[ \chi_2(g,t)=
(D_{X_2}\phi)(gX_1(t))
-(D_{X_2}\phi)(g)+\chi_1^{X_1}(g,t)+\chi_1^{X_2}(gX_1(t),t)\] 
is a continuous function on $G \times \R$ vanishing 
in all pairs~$(g,0)$. 
\end{prf}

Based on the preceding lemma, we 
obtain Lemma \ref{lem:2.6} below, which is  a 
sharpening  
of \cite[Lemma~2.3]{BB11}.
Here the main point is that Lemmas~\ref{lem:2.2} and \ref{lem:2.3} 
permit us to draw stronger conclusions from the proof given 
in \cite{BB11}.

\begin{lem} \mlabel{lem:2.6}
Let $G$ be a topological group and $V$ be a locally convex space. 
Let $X, X_1, X_2 \in \fL(G)$ with $X = X_1 + X_2$ and 
$\phi \in \cC(G,V)$ be such that 
$D_{X_1} \phi$ and $D_{X_2} \phi$ exist and are continuous. 
Then $D_X\phi$ also exists and satisfies 
\begin{equation}
  \label{eq:sum} 
D_X \phi  = D_{X_1} \phi + D_{X_2} \phi.
\end{equation}
\end{lem}

\begin{prf} For $t \in \R$ we put $\gamma(t)=X_1(t)X_2(t)$. 
Fix $g_0 \in G$. We have to show that $D_X\phi$ is defined in $g_0$ 
and satisfies  
\begin{equation}\label{eq:4} 
(D_X\phi)(g_0)
=(D_{X_1}\phi)(g_0)+(D_{X_2}\phi)(g_0).
\end{equation} 
To this end fix an arbitrary continuous seminorm 
$\vert\cdot\vert$ on $V$ and let $\eps>0$ be arbitrary. 
Since the functions $D_{X_j}\phi$ are continuous for $j=1,2$, 
there exists an open $\1$-neighborhood $U$ in $G$ such that 
\begin{equation} \label{eq:5} 
(\forall g\in U)\quad 
\vert((D_{X_1}+D_{X_2})\phi)(g_0g)
-((D_{X_1}+D_{X_2})\phi)(g_0)\vert<\frac{\eps}{2}.
\end{equation}

Let $\delta_1>0$ be such that $X(t)\in U$ for $-\delta_1\leq t \leq \delta_1$. 
By using \eqref{eq:trotter} with uniform convergence on the interval 
$[-\delta_1,\delta_1]$, we obtain an $n_1 \in \N$ such that 
if $n\ge n_1$ and $-\delta_1\le t\le\delta_1$, 
then $\gamma(t/n)^n\in U$. There also exists a $\delta_2\in(0,\delta_1)$ such that 
for  $n= 1,\dots,n_1$ and $-\delta_2\le t\le\delta_2$ we have 
 $\gamma(t/n)^n\in U$. 
Therefore 
\[ \gamma\Bigl(\frac{t}{n}\Bigr)^n\in U
\text{ if } -\delta_2\le t\le\delta_2 
\text{ and }n\in \N. \] 
If  $1\le k\le n$ and $\vert t\vert\le\delta_2$, 
then $\vert(k/n)t\vert\le\vert t\vert\le\delta_2$, 
so that we also obtain 
$$\gamma\Bigl(\frac{t}{n}\Bigr)^k
=\gamma\Bigl(\frac{(k/n)t}{k}\Bigr)^k
\in U.$$ 

According to Lemma~\ref{lem:2.2}, the set 
\[ C := X([-\delta_2, \delta_2]) \cup 
\bigcup_{n \in \N} 
\Big\{ \gamma\Big(\frac{t}{n}\Big)^n \: -\delta_2 \leq t \leq \delta_2\Big\} \] 
is compact. 
Therefore the continuity of 
$\chi_2$ in Lemma~\ref{lem:2.5}(ii) implies (by Lemma \ref{lem:2.3})
the existence of $\delta_3 \in (0,\delta_2)$ 
such that, for $x \in g_0C$ and $0 < |t| \leq \delta_3$, we have 
\begin{equation}\label{eq:6}
|\chi_2(x,t)| = \Bigl\vert\frac{1}{t}(\phi(x\gamma(t))-\phi(x))
-((D_{X_1}+D_{X_2})\phi)(x)\Bigr\vert
<\frac{\eps}{2}. 
\end{equation}
Note that we also have 
$\gamma(t/n)^k \in C$ for $0 \leq k \leq n$ and $|t| \leq \delta_2$.  
This allows us to use \eqref{eq:5} 
in order to show that if $0 < |t| \leq \delta_3$ and $n\in \N$, then 
\allowdisplaybreaks
\begin{align}
&\Bigl\vert\frac{1}{t}
\Bigl(\phi\Bigl(g_0 \gamma\Bigl(\frac{t}{n}\Bigr)^n\Bigr)-\phi(g_0)\Bigr)
-((D_{X_1}+D_{X_2})\phi)(g_0)\Bigr\vert  \nonumber \\
&\le\frac{1}{n}\sum_{k=1}^n
\Bigl\vert\frac{1}{t/n}
\Bigl(
\phi\Bigl(g_0 \gamma\Bigl(\frac{t}{n}\Bigr)^k\Bigr)
-\phi\Bigl(g_0 \gamma\Bigl(\frac{t}{n}\Bigr)^{k-1}\Bigr)
\Bigr) 
\hskip-3pt 
-((D_{X_1}+D_{X_2})\phi)\Bigl(g_0 \gamma\Bigl(\frac{t}{n}\Bigr)^{k-1}\Bigr)
\Bigr\vert \nonumber \\
&\hskip10pt +\frac{1}{n}\sum_{k=1}^n
\Bigl\vert 
((D_{X_1}+D_{X_2})\phi)\Bigl(g_0 \gamma\Bigl(\frac{t}{n}\Bigr)^{k-1}\Bigr)
-((D_{X_1}+D_{X_2})\phi)(g_0)
\Bigr\vert 
\nonumber \\
& <\frac{1}{n}\sum_{k=1}^n\frac{\eps}{2}+\frac{1}{n}\sum_{k=1}^n\frac{\eps}{2}=\eps. \nonumber
\end{align}
Since $\phi$ is continuous, we have 
$\lim_{n\to\infty}\phi(g_0 \gamma(t/n)^n)=\phi(g_0X(t))$ in~$V$. 
The above estimates thus lead for $0 < |t| \leq \delta_3$ to 
\[ \Bigl\vert\frac{1}{t}
(\phi(g_0X(t))-\phi(g_0))
-((D_{X_1}+D_{X_2})\phi)(g_0)\Bigr\vert\le\eps.\] 
Since $\vert\cdot\vert$ is an arbitrary continuous seminorm on 
the Hausdorff locally convex space~$V$, 
it follows that $D_X \phi$ exists in $g_0$ and that 
\eqref{eq:4} is satisfied. 
\end{prf}

The following theorem extends \cite[Thm~2.4]{BB11} in the sense that 
it applies to all continuous functions $G \to V$, which facilitates the 
application of this theorem considerably. 

\begin{thm} \mlabel{thm:3.4} Let $G$ be a Lie group with 
the Trotter property. Then, for every locally convex space $V$ and 
$\phi \in \cC(G,V)$ for which $D_x \phi := D_{\gamma_x}\phi$ exists and is continuous 
for every $x \in \L(G)$, the map 
\[ \L(G) \to \cC(G,V), \quad x \mapsto D_x \phi \] 
is linear. 
\end{thm}

\section{Examples of groups with the Trotter property}
\label{sec:examples}
An important point of the Trotter property is that almost all natural 
classes of infinite-dimensional Lie groups have this property, even if they are
not locally exponential, such as groups of diffeomorphisms. 


\begin{lem}
\label{uniformcnv}
Let $G$ and $H$ be topological groups, $f:G\to H$ be a continuous map, and  $\gamma_n:I=[0,1]\to G$ be a sequence of continuous curves converging uniformly to a limit curve $\gamma$. Then the sequence $f\circ \gamma_n$ converges uniformly to $f\circ\gamma$. 
\end{lem} 

\begin{proof}
The set $K:=\gamma(I)$
is compact and the statement follows from
Lemma \ref{lem:2.3}.
\end{proof}

\begin{lem}
\mlabel{bootstrap}
Let $G$ be a topological group and $\gamma:\R\to G$ be a continuous curve with 
$\gamma(0) = \1$. Set 
\[
\gamma_n:\R\to G\ , \gamma_n(t):=\gamma\Big(\frac{t}{n}\Big)^n.
\]
 Assume that there exists an  
$\varepsilon>0$ such that the sequence 
$(\gamma_n\res_{[-\eps,\eps]})_{n = 1}^{\infty}$ converges uniformly.
Then $(\gamma_n)_{n=1}^\infty$ converges uniformly on compact subsets of $\R$.
\end{lem}

\begin{prf}
It suffices to show that the sequence $(\gamma_n)_{n=1}^\infty$ converges uniformly on 
$[-2\eps,2\eps]$, since one can iterate the argument. 
Define 
\[ \eta:[-\varepsilon,\varepsilon]\to G, \quad 
\eta(t):=\lim_{n\to\infty}\gamma_n(t) \] 
and observe that
\[\gamma_{2n}(t) = \gamma_n\Big(\frac{t}{2}\Big)^2. \] 
Since the square map $f:G\to G, g \mapsto g^2$ is continuous, 
from Lemma \ref{uniformcnv} it follows that the sequence 
$(\gamma_{2n})_{n=1}^\infty$ converges uniformly on $[-2\eps,2\eps]$ to 
the curve $t \mapsto \eta\big(\frac{t}{2}\big)^2$. 
 A minor modification of the above argument proves that $(\gamma_{2n+1})_{n=0}^{\infty}$ 
also converges uniformly on $[-2\eps,2\eps]$ to the same curve. In fact, 
\begin{equation}
\label{eq:gamma2n+1}
\gamma_{2n+1}(t) = \gamma\Big(\frac{t}{2n+1}\Big)^{2n+1}= \gamma\Big(\frac{t}{2n+1}\Big)
\gamma_n(h_n(t))^2, \quad \mbox{ where } \quad 
h_n(t):=\frac{nt}{2n+1}.
\end{equation}
Since $\gamma(0) = \1$, the curves 
$t \mapsto \gamma\big(\frac{t}{2n+1}\big)$ 
converge uniformly on $[-2\eps,2\eps]$ 
to the constant function with value $\1$. 
The curves $(\gamma_n\circ h_n)_{n=0}^\infty$ converge uniformly on $[-2\eps,2\eps]$ to 
$\eta\big(\frac{t}{2}\big)$. 
Since the map $
\mu:G\times G\times G\to G$ defined 
by $\mu(x,y,z):= xyz$
is continuous, Lemma \ref{uniformcnv} implies that
$(\gamma_{2n+1})_{n=0}^\infty$
converges uniformly  on  $[-2\eps,2\eps]$ to 
$\eta\big(\frac{t}{2}\big)^2$. 
\end{prf}

The following proposition provides a criterion for a Lie group  to have 
the Trotter property. We shall see below that it is crucial to verify that 
certain groups of diffeomorphisms have the Trotter property. 

\begin{proposition}
\label{fromepstocpt}
Let $G$ be a  Lie group with a smooth exponential map. Assume that, for every $x_1,x_2\in\L(G)$, there exists an $\varepsilon>0$ such that
\begin{align}
\label{expx1x2eq} 
\exp_G(t(x_1 + x_2)) = \lim_{n \to \infty} 
\Big(\exp_G\Big(\frac{t}{n}x_1\Big)\exp_G\Big(\frac{t}{n}x_2\Big)\Big)^n \end{align} 
holds uniformly on $[-\varepsilon,\varepsilon]$. Then 
$G$ has the Trotter property.
\end{proposition}
\begin{prf}
Setting $\gamma:\R\to G\,,\,\gamma(t):=\exp_G({tx_1})\exp_G({tx_2})$ in Lemma \ref{bootstrap} implies that 
the right hand side of \eqref{expx1x2eq} 
is uniformly convergent  on
compact subsets of $\R$
to a map $\eta:\R\to G$. 
For every $t\in\R$ and $\gamma_n(t) := \gamma\big(\frac{t}{n}\big)^n$ 
we have 
\[
\eta(2t)
=\lim_{n\to\infty}\gamma_{2n}(2t) 
=\lim_{n\to\infty}\gamma_{n}(t)^2 
=\eta(t)^2.\] 
Since
$\eta(t)=\exp_G(t(x_1+x_2))$ for $t\in[-\varepsilon,\varepsilon]$, 
it follows that 
$\eta(t)=\exp_G(t(x_1+x_2))$ for every $t\in\R$.
\end{prf}

\subsubsection*{Covering groups}
Since the condition in Proposition 
\ref{fromepstocpt}
is local,  
we immediately 
obtain the following statement. 

\begin{cor} If $q_G \: \hat G \to G$ is a covering morphism of Lie groups, 
then $G$ has the Trotter property if and only if $\hat G$ has the Trotter property. 
\end{cor}

\subsubsection*{Locally exponential groups}

\begin{proposition}
\mlabel{locexpprp}
Every locally exponential Lie group has the Trotter property.
\end{proposition}

\begin{prf} Let $G$ be a locally exponential Lie group and 
$x,y \in \g = \L(G)$. Then 
\[ \gamma:\R\to G, \quad \gamma(t):=\exp_G({tx})\exp_G({ty}) \] 
is a smooth curve. Since $G$ is locally exponential, 
there exists an $\eps > 0$ and a smooth curve 
$\beta \: [-\eps, \eps] \to \g$ with 
$\gamma(t) = \exp_G({\beta(t)})$ for $|t| \leq \eps$. 
Then $x + y = \gamma'(0) = \beta'(0)$ follows from 
$T_0(\exp_G) = \id$. 
By Taylor's formula 
$\beta(t) = t\delta(t)$ for a continuous function $\delta$ on $[-\eps,\eps]$ 
satisfying $\delta(0) = \beta'(0)$. 
Then 
\[ \beta_n(t) := n\beta\Big(\frac{t}{n}\Big)= t\delta\Big(\frac{t}{n}\Big) \] 
converges uniformly on every compact subset of $\R$ to 
the curve $\eta(t) := t(x+y)$. 
Considering $(\g,+)$ as a topological group, Lemma \ref{uniformcnv} implies that 
the sequence $\gamma_n = \exp_G \circ \beta_n$ converges uniformly on compact subsets of $\R$
to $\exp_G\circ\eta$. 
\end{prf}

\begin{ex}
\mlabel{ex:3.6}
Proposition
\ref{locexpprp} implies that mapping groups, and in particular loop groups, have the Trotter property.  
If $M$ is a compact manifold and $K$ is a locally exponential Lie group 
with the Lie algebra $\fk$, then it follows from \cite[Th.~IV.1.12]{Ne06} that,
for $r \in \N_0 \cup \{\infty\}$, the 
mapping group $\cC^r(M,K)$ is a locally exponential 
Lie group with Lie algebra $\cC^r(M,\fk)$. 
Note also that central extensions of locally exponential Lie groups are 
locally exponential by \cite[Th.~IV.2.11]{Ne06}.  
\end{ex}

\subsubsection*{Diffeomorphism groups}
Our next goal is to prove that groups of automorphisms 
of principal bundles on compact smooth manifolds have the Trotter property. To this end, we begin by a closer look at the topology of these groups. 

If $X, Y$ are topological spaces, then we write 
$\cC(X,Y)$ for the space of continuous maps $X \to Y$, endowed 
with the compact open topology. Suppose that $X$ is locally compact. 
Then the group $\Homeo(X)$, endowed with 
the topology inherited from the embedding 
\[ \Homeo(X) \to \cC(X,X)^2, \quad \phi \mapsto 
(\phi, \phi^{-1}), \] 
is a topological group 
(\cite[Cor.~9.15]{Str06}). If, in addition, $X$ is a $\cC^k$-manifold 
for some $k \in \N$,  
then we endow the group $\Diff^k(M)$ of $\cC^k$-diffeomorphisms with the 
group topology inherited from the embedding 
\[ T^k \: \Diff^k(X) \to \Homeo(T^k(X)), \quad 
\phi \mapsto T^k(\phi).\]
This topology is called the {\it modified compact open 
$\cC^k$-topology}. If $X$ is a smooth manifold, then 
we endow the group $\Diff(X) = \Diff^\infty(X)$ with the group topology inherited from 
the embedding
\[  \Diff(X) \to \prod_{k \in \N_0} \Homeo(T^k(X)), \quad 
\phi \mapsto (T^k(\phi))_{k \in \N_0}\]
where $\N:=\{1,2,3,\ldots\}$ and $\N_0=\N\cup\{0\}$. 
This topology is called the {\it modified smooth compact open topology}. 

In general, this topology does not turn $\Diff(X)$ into a Lie 
group. A typical example is the disjoint union 
$X = \coprod_{n \in \N} \bS^1$ of infinitely many circles (in 
this case the group $\Diff(X)$ contains
the compact group $\T^\N$ as a topological subgroup, and this is not 
compatible with a manifold structure). 
However, there are many important cases where 
topological subgroups of $\Diff(X)$ actually carry natural Lie group 
structures.

\begin{exs} \mlabel{ex:diff} The following groups carry Lie group 
structures compatible with their modified smooth compact open topology. 

(a) If $M$ is a compact manifold, then $\Diff(M)$ carries a natural Lie group 
structure (\cite{Ha82}), \cite{Mil84}, \cite{Ne06}). 
From the smoothness of the action of the Lie group $\Diff(M)$ 
on $M$, one easily derives that the modified smooth compact open topology 
is coarser than the Lie topology, but the construction of charts for the manifold structure 
on $\Diff(M)$ (see \cite[Ex. II.3.14]{Ne06}) easily implies that the two topologies 
coincide. 

(b) The group $\Aut(P) = \Diff(P)^K$ of automorphisms of a 
principal $K$-bundle $q \: P \to M$ over a compact smooth manifold $M$. 
It is a Lie group extension of the Lie group $\Diff(M)$  
by the gauge group $\Gau(P)$ (\cite{Wo07}, \cite{ACMM89}). 

(c) The group $\Aut(\bV)$ of automorphisms of a vector bundle 
over a compact smooth manifold can be identified 
with the automorphism group of the corresponding frame bundle 
$\Fr(\bV)$, which is a principal $\GL(V)$-bundle. 
Therefore (b) applies to $\Aut(\bV)$. 

(d) The higher tangent bundles $T^k(M)$ of a compact smooth manifold 
carry natural bundle structures, which, for $k > 1$, are not vector 
bundles. They are {\it multilinear bundles} (cf.\ \cite[Sect.~15]{Be08}).  
This implies that they are also associated to a principal 
$K$-bundle $q_P \:  P \to M$, where $K$ 
is a finite-dimensional Lie group of polynomial 
diffeomorphisms of the fiber, such that 
$\Aut(T^k(M)) \cong \Aut(P)$. As in (c), we thus obtain 
on $\Aut(T^k(M))$ a Lie group structure from (b). 
\end{exs}

\begin{rem} \mlabel{rem:vecflow} Let $\cV(M)$ denote the space of smooth vector 
fields on the smooth manifold $M$. If 
$X \in \cV(M)$ is complete, we write 
$(\Phi^X_t)_{t \in \R}$ for the corresponding smooth flow. 
Then the smoothness of the flow map 
$\Phi^X \: \R \times M \to M$ and $\Phi^X_{-t} = (\Phi^X_t)^{-1}$ 
implies that \break $\Phi^X \: \R \to \Diff(M)$ is a continuous 
one-parameter group with respect to the modified 
smooth compact open topology. Here we use that the induced 
flow $T^k(\Phi^X_t) = \Phi^{T^k(X)}_t$ on $T^k(M)$ is also smooth. 

If $M$ is compact, then $\cV(M) = \L(\Diff(M))$ is the Lie algebra 
of the group of all diffeomorphisms of $M$ and the exponential 
function is given by the time-$1$-flow $\exp(X) = \Phi^X_1$.  
Note that the compactness of $M$ implies that 
every vector field on $M$ is complete. 

Similarly, the Lie algebra $\aut(P)$ of the automorphism 
group $\Aut(P)$ of a principal $K$-bundle over $M$ 
is the Lie algebra $\cV(P)^K$ of $K$-invariant vector fields. 
All these vector fields are complete and 
$\exp(X) = \Phi^X_1$ defines the exponential function 
$\aut(P) \to \Aut(P)$. 
\end{rem}

\begin{rem} 
\label{rem:3.9}

(a) For the following we recall from 
Theorem 1 in \cite[Ch.~2, \S 4.1]{Bou89} that 
every compact space $X$ carries a unique uniform structure compatible 
with the topology. 

(b) We also recall that, for a locally compact space $X$ 
and a uniform space $Y$, on $\cC(X,Y)$ the topology of uniform 
convergence on compact subsets coincides with the compact open topology  
(cf.\ \cite[Ch.~10]{Bou74}). 

(c) Putting both pieces together, we see that, 
if $X$ is a compact space and $Y$ a compact subset of a topological space 
$Z$, then the compact open topology on the subspace 
$\cC(X,Y) \subeq \cC(X,Z)$ coincides with the topology of uniform convergence.  
This means that, for a sequence $f_n \: X \to Z$ of continuous map 
for which $\bigcup_{n \in \N} f_n(X)$ has compact closure $Y$, 
uniform convergence in $\cC(X,Y)$ is equivalent to convergence in 
$\cC(X,Z)$ in the compact open topology. 
\end{rem}

\begin{lem} \mlabel{lem:nelson} 
Let $M$ be a smooth manifold and 
$X, Y \in \cV(M)$ be complete vector fields for which 
$X + Y$ is also complete. Then  for each 
compact subset $C \subeq M$ 
there exists
an $\eps > 0$ such that 
\begin{equation}
\mlabel{phiphixphiy}
 \Phi^{X+Y}_t = \lim_{n \to \infty} \Big( \Phi^X_{t/n} \circ \Phi^Y_{t/n}\Big)^n 
 \end{equation} 
holds uniformly on $|t| \leq \eps$, where 
both sides are considered as 
functions from $\R$ into $\cC(C,M)$, endowed with the compact open topology.

\end{lem}

\begin{prf} From \cite[\S 4, Th.~1]{Nel69} it follows that, 
every point $m_0 \in M$ has a neighborhood $V$ for which there exists an 
$\eps_V > 0$ such that 
$\gamma_n(t) := \big( \Phi^X_{t/n} \circ \Phi^Y_{t/n}\big)^n$ converges 
to $\Phi^{X+Y}_t$ uniformly on $V$ and for $|t| \leq \eps_V$. 
Covering $C$ with finitely many such neighborhoods 
$V_1, \ldots, V_N$, we put $\eps := \min \{ \eps_{V_1}, \ldots, \eps_{V_N}\}$. 
For $|t| \leq \eps$ and on $C$,
both sides of \eqref{phiphixphiy}
attain values 
in some compact subset $D \subeq M$. By Remark 
\ref{rem:3.9}
the topology of uniform convergence on
 $\cC(C,D)$ coincides with the compact open topology 
induced  
 from $\cC(C,M)$.
 \end{prf}

\begin{thm} 
\label{thm:prinbun} 
Let $K$ be a finite-dimensional Lie group 
and $q \: P \to M$ be a smooth $K$-principal bundle 
over the compact smooth manifold $M$. 
Then the Lie group $\Aut(P) = \Diff(P)^K$ of bundle automorphisms 
has the Trotter property.
\end{thm}

\begin{prf}  Let $X, Y \in \aut(P) = \cV(P)^K$. We need to show that 
\begin{equation}
  \label{eq:trot-vec2}
\Phi^{X+Y}_t = \lim_{n \to \infty} \Big( \Phi^X_{t/n} \circ \Phi^Y_{t/n}\Big)^n
\end{equation}
holds uniformly on compact subsets of $\R$ with respect to the 
modified smooth compact open topology on $\Aut(P)$. 
We thus have to show that, for every $r \in \N_0$, 
we have 
\begin{equation}
  \label{eq:trot-vec3}
\Phi^{T^r(X)+T^r(Y)}_t = \lim_{n \to \infty} \Big( \Phi^{T^r(X)}_{t/n} 
\circ \Phi^{T^r(Y)}_{t/n}\Big)^n
\end{equation}
uniformly on compact subsets of $\R$ in 
the topological group $\Homeo(T^r(P))$ with respect to the 
modified compact open topology. 
Here  $T^r(X) \: T^r(P) \to T^{r+1}(P)$ denotes the canonical lift 
of $X \in \aut(P)$ to a smooth vector field on $T^r(P)$ which 
generates the flow 
$T^r(\Phi^X_t) = \Phi^{T^r(X)}_t$ on $T^r(P)$. 


From Lemma~\ref{lem:nelson} it follows that, 
for every compact subset $C \subeq P$, there exists an $\eps > 0$ such that 
\[ \Phi^{X+Y}_t\res_C = \lim_{n \to \infty} \Big( \Phi^X_{t/n} \circ \Phi^Y_{t/n}\Big)^n
\res_C \] 
holds uniformly on $[-\eps, \eps]$ in the space $\cC(C,P)$, endowed 
with the compact open topology. 
Since 
$M$ is compact, there exists a compact subset 
$C \subeq P$ whose interior $C^0$ satisfies $q_P(C^0) = M$. 
We now have $P = C^0\cdot K$, and for every compact subset 
$C' \subeq P$ there exists a finite subset $F \subeq K$ 
with $C' \subeq C\cdot F$. This implies that 
\[ \Phi^{X+Y}_t = \lim_{n \to \infty} \Big( \Phi^X_{t/n} \circ \Phi^Y_{t/n}\Big)^n\] 
actually holds uniformly on $[-\eps, \eps]$ in $\cC(P,P)$.

For $r \in \N$, let 
$C^r \subeq T^r(P)$ be a compact subset which is a neighborhood 
of $C$ in $T^r(P)$ (recall that $C\subeq P \subeq T^r(P)$). 
According to Lemma~\ref{lem:nelson}, there exists an $\eps_r > 0$ such that 
\[ \Phi^{T^r(X)+T^r(Y)}_t\res_{C^r} 
= \lim_{n \to \infty} \Big( \Phi^{T^r(X)}_{t/n} \circ \Phi^{T^r(Y)}_{t/n}\Big)^n
\res_{C^r} \] 
holds uniformly on $[-\eps_r, \eps_r]$ in the space $\cC(C^r,T^r(P))$, endowed 
with the compact open topology. As above, we see that the 
same statement holds with $\tilde C := (C^r)^0\cdot K$ instead of $C^r$ 
(here we use the canonical lift of the $K$-action on $P$ to $T^k(P)$). 
This is an open subset of $T^r(P)$ containing the canonical 
image of $P$. As $T^r(P)$ is associated to a principal 
bundle with some structure group $H$ over $P$ 
(cf.\ Example~\ref{ex:diff}(d)) and 
the maps $T^r(\phi)$, $\phi \in \Diff(P)$, are bundle
automorphisms, it follows with a similar argument as above, 
applied to the flows on the corresponding principal $H$-bundle, 
that actually 
\begin{equation}
  \label{eq:trot-vec} 
 \Phi^{T^r(X)+T^r(Y)}_t 
= \lim_{n \to \infty} \Big( \Phi^{T^r(X)}_{t/n} \circ \Phi^{T^r(Y)}_{t/n}\Big)^n
\end{equation}
holds uniformly on $[-\eps_r, \eps_r]$ in the space $\cC(T^r(P),T^r(P))$, endowed 
with the compact open topology. 

Since 
\[ (\Phi^{T^r(X)+T^r(Y)}_t)^{-1} 
= \Phi^{T^r(X)+T^r(Y)}_{-t} 
= \lim_{n \to \infty} \Big( \Phi^{T^r(Y)}_{-t/n} \circ \Phi^{T^r(X)}_{-t/n}\Big)^n, \]  
we can also apply the preceding argument with 
$X$ and $Y$ exchanged to see that, for some $0<\eps'_r < \eps_r$, 
 we have uniform convergence in $[-\eps'_r, \eps'_r]$ 
of \eqref{eq:trot-vec} 
 with respect to the 
modified compact open topology on the group $\Homeo(T^r(P))$. 
Lemma \ref{bootstrap} now implies that \eqref{eq:trot-vec} 
holds uniformly on compact subsets of $\R$.
\end{prf}

\begin{cor} \mlabel{cor:3.11} If $M$ is a compact smooth manifold, then the Lie group 
$\Diff(M)$ of smooth diffeomorphisms of $M$ has the Trotter property.
\end{cor}

\subsubsection*{Direct limits} 

\begin{ex} \mlabel{ex:2.8} 
If $G$ is the direct limit of a sequence $(G_n)$ of 
finite-dimensional Lie groups and  injective homomorphisms $G_n \into G_{n+1}$, 
then $\L(G)$ is isomorphic as a topological Lie algebra 
to the corresponding direct limit of finite-dimensional Lie algebras 
$\L(G_n)$ and $G$ has the Trotter property (cf.\ \cite[Prop.~4.6]{Gl05}). 
\end{ex} 

\subsubsection*{Semidirect products} 

Let $V$ be a complete locally convex space, $G$ a Lie group with a smooth 
exponential function and 
$\alpha \: G \to \GL(V)$ be a homomorphism defining a smooth action 
of $G$ on $V$, so that we can form the semidirect product 
Lie group $V \rtimes_\alpha G$. This Lie group has a smooth exponential 
function, given explicitly by 
\[ \exp_{V \rtimes_\alpha G}(v,x) 
= \big( \beta(x)v,  \exp_G(x)\big) \quad \mbox{ with } \quad 
\beta(x) =  \int_0^1 \alpha(\exp_G(sx))\, ds\] 
(\cite[Ex.~II.5.9]{Ne06}). 

\begin{prop} \mlabel{prop:trot-semdir} If $G$ has the Trotter property, 
then so does $H := V \rtimes_\alpha G$. 
\end{prop}

\begin{prf} Let $X = (u,x), Y = (w,y) \in \L(H) \cong V \rtimes \g$, 
$z := x + y$, 
and consider the smooth curves 
\[ \Gamma(t) := \exp_H(tX) \exp_H(tY) 
\quad \mbox{ and } \quad 
\Gamma_n(t) := \Gamma(t/n)^n \] 
in $H$. In terms of the semidirect product structure, 
we write $\Gamma(t) = (\delta(t), \gamma(t))$ where
\[ 
\delta(t) = t \beta(tx)u + \alpha(\exp_G(tx))t \beta(ty)w\,
 \text{ and } \,
\gamma(t) = \exp_G(tx) \exp_G(ty).\] 
If $G$ has the Trotter property, then 
$\gamma_n(t) = \gamma(t/n)^n$ converges uniformly on each compact 
interval in $\R$ to $\exp_G(tz)$, the second component 
of $\exp_H(t(X + Y))$. To see what happens in the first component, 
we note that 
\[ \gamma_n(t) = \Big(\delta(t/n), \gamma(t/n)\Big)^n  
= \Big(B_n(t)n\delta(t/n), \gamma(t/n)^n\Big),\] 
where 
\[ B_n(t) = \frac{1}{n}\Big(\1 
+ \alpha(\gamma(t/n)) + \cdots + 
\alpha(\gamma(t/n))^{n-1}\Big).\] 

Next we observe that $\delta \: \R \to V$ is a smooth curve 
with $\delta(0) = 0$ and $\delta'(0) = u + w$. Therefore 
$n\delta(t/n)$ converges uniformly on compact subsets of $\R$ 
to $t(u + w)$. It therefore suffices to show that, on 
compacts subsets $C_V \subeq V$, we have 
\[ B_n(t) \to \beta(t(x+y)) = \beta(tz) 
= \int_0^1 \alpha(\exp_G(stz))\, ds.\] 

Fix $T > 0$. Writing 
\begin{align*}
B_n(t)v - \beta(tz)v 
&= B_n(t)v - \int_0^1 \alpha(\exp_G(stz))\, ds \\ 
&= \frac{1}{n}\Big(
\sum_{j = 0}^{n-1} \alpha(\gamma(t/n))^j v - \alpha\Big(\exp_G\Big(\frac{jt}{n}z\Big)\Big)v\Big) \\ 
&\qquad\qquad+ \frac{1}{n} \sum_{j = 0}^{n-1} \alpha\Big(\exp_G\Big(\frac{jt}{n}z\Big)\Big)v
- \int_0^1 \alpha(\exp_G(stz))v\, ds,  
\end{align*}
we see that the second summand converges to $0$ because it describes a 
Riemann sum approximation of the integral $\beta(tz)v$, which converges 
uniformly for $|t| \leq T$ and $v \in C_V$. 

By Lemma \ref{lem:2.2}, for every compact subset $C \subeq V$, the set 
\[ \{ \alpha(\gamma_n(t))v \: v \in C, |t| \leq T, n \in \N \} \] 
has compact closure $\tilde C$ which contains also the elements 
$\alpha(\exp_G(tz))v$, $v \in C$, $|t| \leq T$. 

Now let $p \: V \to \R$ be a continuous seminorm 
and $C \subeq V$ be a compact subset. 
Then  for every $\eps > 0$ 
there exists
a $\1$-neighborhood $U= U_{\eps,C} \subeq G$ such that 
$p(\alpha(g)c - c) \leq \eps$ for $g \in U_{\eps,C}$, 
$c \in \tilde C$. 

Since 
$\gamma_n(t) \to \exp_G(tz)$ uniformly on $[-T,T]$, 
 for each $\1$-neighborhood $U \subeq G$
there exists  
an element $N_U \in \N$ such that 
\[ \gamma_n(t) \in U\exp_G(tz) \quad \mbox{ for } \quad n \geq N_U, 
|t| \leq T.\] 
For $0 \leq j < n$ we then have 
\[ \gamma(t/n)^j  = \gamma(tj/nj)^j = \gamma_j(tj/n) 
\in U \exp_G\Big(\frac{tj}{n}z\Big) \quad \mbox{ for } \quad j \geq N_U.\] 
For $j \geq N_U$ we thus obtain for $v \in C$ the estimate 
\[ p\Big(\alpha(\gamma(t/n))^j v - \alpha(\exp_G\big(\frac{jt}{n}z\big))v\Big) 
\leq \eps\] 
and therefore 
\[ p\Big(\frac{1}{n}\Big(
\sum_{j = 0}^{n-1} \alpha(\gamma(t/n))^j v - \alpha\Big(\exp_G\Big(\frac{jt}{n}z\Big)\Big)v
\Big) 
 \leq \frac{N_U}{n} 2 \sup p(\tilde C) + \frac{n-N_U}{n} \eps
< 2 \eps \] 
if $n$ is sufficiently large. This proves that 
$B_n(t)v \to \beta(tz)v$ uniformly for $|t| \leq T$ and $v \in C$ 
and hence that $H = V \rtimes_\alpha G$ has the Trotter property. 
\end{prf}

\subsubsection*{Regular Lie groups}
For the definition and properties of $C^0$-regularity, which  are used in the next theorem and its proof, see 
Appendix  \ref{app:b}.
\begin{thm} \mlabel{thm:3.15}
Suppose that $G$ is a $C^0$-regular Lie group with the Trotter 
property and that $q\: \hat G \to G$ is a central extension by  
a $C^0$-regular abelian Lie group $Z$ with Lie algebra $\fz$ 
(f.i., $Z = \fz/\Gamma$ where $\fz$ is complete and $\Gamma \subeq \fz$ is 
a discrete subgroup). 
Then $\hat G$ also has the Trotter property. 
\end{thm}

\begin{prf}
Since $C^0$-regularity is an extension property by 
Theorem~\ref{thm:reg-ext-prop}, the group $\hat G$ is $C^0$-regular. 
We want to show that $\hat G$ also has the Trotter property. 
So let $\hat x, \hat y \in \hat\g$, 
$x := \L(q)\hat x$ and $y := \L(q)\hat y$. 
We also put $z := x + y$ and $\hat z := \hat x + \hat y$. 
We consider the smooth curves 
\[ \hat\gamma(t) := \exp_{\hat G}(t\hat x) \exp_{\hat G}(t\hat y) 
\quad \mbox{ and } \quad 
\gamma(t) := \exp_G(t x) \exp_G(t y) = q(\hat\gamma(t)).\] 

To see that $\hat\gamma_n(t) := \hat\gamma(t/n)^n$ converges 
uniformly on compact subsets of $\R$ to the curve $\exp_{\hat G}(t(\hat x+\hat y))$, 
it suffices to show that the corresponding sequence of 
(left) logarithmic derivatives $\delta(\hat\gamma_n)$ 
converges uniformly on compact subsets to the constant curve $\hat x + \hat y$. 

From the product rule 
$\delta(\alpha\beta) = \delta(\beta) + \Ad(\beta)^{-1} \delta(\alpha)$ 
it follows that 
\begin{align*}
 \delta(\hat\gamma_n)
& = \frac{1}{n}\delta(\hat\gamma)_{t/n} + \Ad(\hat\gamma(t/n))^{-1}
 \frac{1}{n}\delta(\hat\gamma)_{t/n} + \cdots + \Ad(\hat\gamma(t/n))^{-n+1}
 \frac{1}{n}\delta(\hat\gamma)_{t/n}  \\
 &= \frac{1}{n}\Big( \sum_{j = 0}^{n-1} \Ad(\hat\gamma(t/n))^{-j}\delta(\hat\gamma)_{t/n}\Big).
\end{align*}
The adjoint action of $\hat G$ on its Lie algebra 
$\hat\g$ factors through an action 
$\hat\Ad \: G \to \Aut(\hat\g)$, so that 
\[  \delta(\hat\gamma_n)
= \frac{1}{n}\Big( \sum_{j = 0}^{n-1} \hat\Ad(\gamma(t/n))^{-j}
\delta(\hat\gamma)_{t/n}\Big).
\] 
We also obtain from the product rule that 
\[ \delta(\hat\gamma)_t 
= \hat y + \Ad(\exp_{\hat G}(t\hat y))^{-1}\hat x 
= \hat y + \hat\Ad(\exp_{G}(-ty)) \hat x,\] 
so that 
\[ \delta(\hat\gamma)_{t/n} \to \hat x + \hat y = \hat z\] 
holds uniformly on compact intervals of $\R$.

As $\gamma_n(t) \to \exp_G(tz)$ uniformly on compact intervals in $\R$, 
we obtain as in the proof of Proposition~\ref{prop:trot-semdir} that 
\begin{align*}
&\lim_{n \to \infty} \frac{1}{n}\Big( \sum_{j = 0}^{n-1} \hat\Ad(\gamma(t/n))^{-j}
\delta(\hat\gamma)_{t/n}\Big) 
=  \lim_{n \to \infty} \frac{1}{n}\Big( \sum_{j = 0}^{n-1} \hat\Ad(\gamma(t/n))^{-j} \hat z \Big) \\
&=  \lim_{n \to \infty} \frac{1}{n}\Big( \sum_{j = 0}^{n-1} \hat\Ad\big(\exp_G (\frac{-jt}{n}z)\big)\hat z \Big) 
=  \lim_{n \to \infty} \frac{1}{n}\Big( \sum_{j = 0}^{n-1}
\hat z \Big) = \hat z.
\end{align*}
This means that $\delta(\hat\gamma_n) \to \hat z$ in the space 
$\cC(\R,\hat\g)$. Since $\hat G$ is $\cC^0$-regular, this implies that 
\[ \hat\gamma_n = \Evol_{\hat G}(\delta(\hat\gamma_n)) \to \Evol_{\hat G}(\hat z), \] 
which is the Trotter--Formula. 
\end{prf}

\begin{thm} {\rm(Gl\"ockner; \cite{Gl12b}; see \cite{OMYK82} for the compact case)} 
For every finite-dimensional smooth manifold $M$, 
the group $\Diff_c(M)$ of compactly supported diffeomorphisms 
is $C^0$-regular. 
\end{thm}

Combining Gl\"ockner's theorem with Corollary~\ref{cor:3.11} 
and Theorem~\ref{thm:3.15}, we obtain: 

\begin{cor} Every central extension of a diffeomorphism group 
$\Diff(M)$ of a compact smooth manifold by a finite-dimensional center has 
the Trotter property. In particular, the Virasoro group $\Vir$ has 
the Trotter property. 
\end{cor}

\section{$\cC^k$-vectors}  \mlabel{sec:3} 

In this section, $G$ denotes a Lie group with a smooth exponential 
function 
\[ \exp_G \: \L(G) = \g\to G.\] 

\begin{defn} \label{def:domains}
Let $(\pi, V)$ be a representation of the Lie group $G$ 
(with a smooth exponential function) on the locally convex 
space~$V$.

(a) We say that $\pi$ is {\it continuous} if the action of $G$ on $V$ 
defined by $(g,v) \mapsto \pi(g)v$ is continuous. 

(b) An element $v \in V$ is a $\cC^k$-vector, $k \in \N_0 \cup \{\infty\}$,  
if the orbit map $\pi^v \: G \to V, g \mapsto \pi(g)v$ is a 
$\cC^k$-map. We write $V^k := V^k(\pi)$ for the linear subspace of 
$\cC^k$-vectors and we say that the representation 
$\pi$ is {\it smooth} if the space $V^\infty$ of smooth vectors 
is dense.  A vector $v\in V$ is called an \emph{analytic} vector if the orbit map $\pi^v$ is analytic. The space of analytic vectors is denoted by $V^\omega$.

(c) For each $x \in \g$, we write 
$$\cD_x := \Big\{ v \in V \: \derat0 \pi(\exp_G (tx))v \ \mbox{ exists } \Big\}
$$ 
for the domain of the infinitesimal generator 
$$\oline{\dd\pi}(x)v := \derat0 \pi(\exp_G(tx))v
$$ 
of the one-parameter group $\pi(\exp_G(tx))$, $\cD^1 := \bigcap_{x \in \g} \cD_x$ 
and $\omega_v(x) := \oline{\dd\pi}(x)v$ for $v \in \cD^1$. 
Each $\cD_x$ and therefore also $\cD^1$ are linear subspaces of~$V$, 
but at this point we do not know whether $\omega_v$ is linear 
(cf.~\cite[Thm.~8.2]{Ne10} for a positive answer for Banach--Lie groups). 

(d) We define inductively  
$$ \cD^n := 
\{ v \in \cD^1 \: (\forall x \in \g)\, \oline{\dd\pi}(x)v \in 
\cD^{n-1}\} \quad \mbox{ for } \quad n > 1, $$
so that 
$$ \omega_v^n(x_1, \ldots, x_n) := 
\oline{\dd\pi}(x_1)\cdots \oline{\dd\pi}(x_n)v$$ 
is defined for $v \in \cD^n$ and $x_1,\ldots, x_n \in \g$. 
We further put $\cD^\infty := \bigcap_{n \in \N} \cD^n$. 
\end{defn}

\begin{rem} \label{rem:cn-vec} 
(a) For every representation $(\pi,V)$ we have 
$$ V^1(\pi) \subeq \cD^1 \quad \mbox{ and } \quad 
V^k(\pi) \subeq \cD^k \quad \mbox{ for } \quad k \in \N. $$
Note that $\omega_v^k$ is continuous and $k$-linear for 
every $v \in V^k(\pi)$. 

(b) By definition, we have $\cD^2 \subeq \cD^1$, so that we obtain 
by induction that $\cD^{n+1} \subeq \cD^n$ for every 
$n \in \N$. 
\end{rem}

(c) If $v\in V^k(\pi)$ then from the continuity of the action $G\times V\to V$ it follows that the map 
$G\times\g ^k\to V\,,\,
(g,x_1,\ldots,x_k)\mapsto \pi(g)\oline{\dd\pi}(x_1)\cdots\oline{\dd\pi}(x_k)v$ is continuous.

The following lemma (\cite[Lemma~3.3]{Ne10}) provides a criterion 
for $\cC^1$-vectors. 

\begin{lem} \mlabel{lem:a.11c} Suppose that $(\pi, V)$ 
is a continuous representation of the Lie group $G$ 
on $V$. Then a vector $v \in \cD^1$ 
is a $\cC^1$-vector if and only if the following two 
conditions are satisfied: 
\begin{description}
\item[\rm(i)] For every smooth curve $\gamma \: [-\eps,\eps]  \to G$, $\eps > 0$,  
with $\gamma(0) = \1$ and $\gamma'(0) = x$, the derivative  
$\frac{d}{dt}|_{t = 0} \pi(\gamma(t))v$ exists and 
equals $\oline{\dd\pi}(x)v$. 
\item[\rm(ii)] $\omega_v \: \g \to V, x \mapsto \oline{\dd\pi}(x)v$ 
is continuous. 
\end{description}

If $G$ is locally exponential, then 
{\rm(i)} follows from {\rm(ii)}. 
\end{lem}

\begin{defn} A Lie group $G$ is called {\it locally $m$-exponential} if there exist 
closed subspaces $\g_1,\ldots, \g_m \subeq \g = \L(G)$ such that 
$\g = \g_1 \oplus \cdots \oplus \g_m $
is a topological direct sum and the map 
\[ M \: \g \to G, \quad M(x_1 + \cdots + x_m) := 
\exp_G (x_1) \cdots \exp_G (x_m) \] 
is a local diffeomorphism in a neighborhood of $(0,\ldots, 0)$. 
\end{defn}

The following lemma extends the implication (ii) $\Rarrow$ (i) of Lemma 
\ref{lem:a.11c}
 to the larger class of locally $m$-exponential groups. 

\begin{lem} \label{lem:a.11d} Suppose that $(\pi, V)$ 
is a continuous representation of the locally $m$-exponential Lie group $G$ 
on $V$. Let $v \in \cD^1$ 
be such that $\omega_v \: \g \to V, x \mapsto \oline{\dd\pi}(x)v$ is continuous and linear. 
Then, for every smooth curve $\gamma \: [-\eps,\eps]  \to G$ 
with $\gamma(0) = \1$ and $\gamma'(0) = x$, the derivative  
$\frac{d}{dt}|_{t = 0} \pi(\gamma(t))v$ exists and 
equals $\oline{\dd\pi}(x)v$. 
In particular, $v$ is a $\cC^1$-vector. 
\end{lem}

\begin{prf} 
Any smooth curve $\gamma$ with $\gamma(0) = \1$ and 
$\gamma'(0) = x$ can be written 
for sufficiently small values of $t \in \R$ as 
$\exp_G(\eta_1(t))\cdots \exp_G(\eta_m(t))$ with smooth curves $\eta_j \: [-\eps,\eps] \to \g$ 
satisfying 
$\eta_j(0) = 0$ and $\sum_j \eta_j'(0) = x$. 
From the proof of \cite[Lem.~3.3]{Ne10} we know that 
for every smooth curve $\eta:[-\eps,\eps]\to \g$ satisfying 
$\eta(0)=0$ we have
$$ 
\deras0 \pi(\exp_G \eta(s))v = \omega_v(\eta'(0)). 
$$
For two smooth curves $\gamma_j \: [-\eps,\eps] \to G$ with 
$\gamma_j(0) = \1$ we have 
\[ 
\pi(\gamma_1(t)\gamma_2(t))v - v 
= 
\pi(\gamma_1(t))\big(\pi(\gamma_2(t))v - v\big) + 
\big(\pi(\gamma_1(t))v - v\big),\] 
and since $G$ acts continuously on $V$, we obtain 
\[ \derat0 \pi(\gamma_1(t)\gamma_2(t))v  
=  \derat0 \pi(\gamma_1(t))v + \derat0 \pi(\gamma_2(t))v.\] 
By induction this leads to 
\begin{align*}
 \derat0 \pi(\gamma(t))v 
&= \sum_{j = 1}^m \derat0 \pi(\exp_G \eta_j(t))v 
= \sum_{j =1}^m \omega_v(\eta_j'(0)) 
= \omega_v\Big(\sum_{j=1}^m \eta_j'(0)\Big) = \omega_v(x). 
\qedhere\end{align*}
\end{prf}

As an immediate consequence of Theorem~\ref{thm:3.4}, we obtain: 

\begin{thm} \mlabel{thm:additive} 
Let $(\pi, V)$ be a continuous representation 
of the Lie group $G$ with the Trotter property on the locally convex space $V$. 
Then, for each $v \in \cD^1$, the map 
$$ \omega_v \: \g \to V, \quad v \mapsto \oline{\dd\pi}(x)v $$
is linear. 
\end{thm}

\begin{prf} For $x \in \g$ and the continuous orbit 
map $\pi^v:G\to V\,,\,\pi^v(g) := \pi(g)v$ we have 
$D_x \pi^v = \pi^{\oline\dd\pi(x)v}$, which exists and is continuous 
for $x \in \g$ and $v \in \cD^1$. Hence the assertion follows from 
Theorem~\ref{thm:3.4}.
\end{prf}

\begin{lem} \mlabel{lem:a.12} {\rm(\cite[Lemma~3.4]{Ne10})} 
If $G$ is locally exponential, then a vector $v \in V$ is a $\cC^k$-vector 
if and only if $v \in \cD^k$ and the 
maps $\omega_v^n$, $n \leq k$, are continuous and $n$-linear. 
In particular, $v$ is a smooth vector if and only if 
$v \in \cD^\infty$ and all the maps $\omega_v^n$ 
are continuous and $n$-linear. 
\end{lem}

\begin{lem} \mlabel{lem:cont-m} Let $V_0, V_1,\ldots, V_n$ be topological 
vector spaces and 
\[ \omega \: V_1 \times \cdots \times V_n \to V_0 \] 
be an $n$-linear map. 
If $\omega$ is continuous in some point 
$(v_1^0,\ldots, v_n^0)$, then $\omega$ is continuous. 
\end{lem}

\begin{prf} 
For $n = 1$ this is obvious. If $n>1$ then we can write
\[ \omega(v_1 - v_1^0, \ldots, v_n - v_n^0) 
= \omega(v_1, \ldots, v_n) \pm \cdots + (-1)^n 
\omega(v_1^0, \ldots, v_n^0). \] 
Each of the summands on the right hand side is a function of $(v_1,\ldots,v_n)$. Continuity of $\omega$ at $(v_1^0,\ldots, v_n^0)$ implies that all 
of these summands  are continuous at  $(v_1^0,\ldots, v_n^0)$. 
We conclude that $\omega$ is continuous at $(0,\ldots, 0)$, and 
it is well-known that this implies the continuity of~$\omega$. 
\end{prf}

\begin{lem}
  \mlabel{lem:contmulti} 
Let $(\pi, V)$ be a continuous representation 
of the Fr\'echet--Lie group $G$ with the Trotter property on the metrizable 
locally convex space $V$. Then, for each $v \in \cD^n$, 
the map 
$$\omega_v^n \: \g^n \to V, \quad 
(x_1,\ldots, x_n) \mapsto \oline{\dd\pi}(x_1)
\oline{\dd\pi}(x_2) \cdots \oline{\dd\pi}(x_n)v $$
is continuous and $n$-linear. 
\end{lem}

\begin{prf} Theorem~\ref{thm:additive}  
implies that $\omega_v^n$ is $n$-linear. 
We argue by induction on $n \in \N_0$ that it is continuous. 
For $n = 0$ this is trivial. 
Now we assume $n>0$ and that $\omega_w^{n-1} \: \g^{n-1} \to V$ 
is a continuous $(n-1)$-linear map for every $w \in \cD^{n-1}$. 

Hence, for $t>0$, the maps $F_t \:  \g^n \to V$, defined by 
\[ F_t(x_1, x_2,\ldots, x_n) := 
\frac{1}{t}\Big(\pi(\exp_G (t x_1)) \omega_v^{n-1}(x_2,\ldots, x_n)- 
\omega_v^{n-1}(x_2,\ldots, x_n)\Big) \] 
are continuous and satisfy 
\[ \lim_{ n \to \infty} F_{\frac{1}{n}}(x_1, x_2,\ldots, x_n) 
= \omega_v^n(x_1,\ldots, x_n).\] 
Since $V$ is metrizable and $\g^n$ is a Baire space, it 
follows from \cite[Ch. IX, \S 5, Ex. 22(a)]{Bou74} that 
the set of discontinuity points of $\omega_v^n$ is of the first category, 
hence not all of $\g^n$. We conclude that there exists a point in which 
$\omega_v^n$ is continuous, so that its continuity follows from 
Lemma~\ref{lem:cont-m}. 
\end{prf}

\begin{thm} \mlabel{thm:2.9}
Let $(\pi, V)$ be a continuous representation 
of the locally exponential Fr\'echet--Lie group $G$ 
on the metrizable locally convex space $V$. 
Then $\cD^k$ coincides with the space of $\cC^k$-vectors for any 
$k \in \N\cup \{\infty\}$. 
\end{thm}

\begin{prf} It only remains to combine 
Lemma~\ref{lem:a.12} with Lemma~\ref{lem:contmulti}. 
\end{prf}

\begin{remark}
The preceding theorem generalizes the corresponding Banach results from \cite{Ne10} considerably 
with substantially simpler proofs.  
\end{remark}
\section{Differentiable vectors for unitary representations} 
\mlabel{sec:4}

The main goal of this section is to prove Theorem~\ref{thm:2.1} which,  for unitary representations, extends 
Theorem~\ref{thm:2.9} to groups with the Trotter property.

In this section 
$G$ is a Lie group with a smooth exponential function. 
We shall see in this section, that we can weaken the assumption of local exponentiality in 
Theorem~\ref{thm:2.9} for unitary representations. 

\begin{defn} Let $(\pi, \cH)$ be a unitary representation of the Lie group $G$. 
We say that $\pi$ is: 
\begin{description}
\item[\rm(i)] {\it differentiable} if the space $\cH^1$ of $\cC^1$-vectors is dense in $\cH$. 
\item[\rm(ii)] {\it smooth} if the space $\cH^\infty$ of smooth vectors is dense in $\cH$. 
\item[\rm(iii)] {\it analytic} if the space $\cH^\omega$ of analytic vectors is dense in $\cH$. 
\end{description}
\end{defn}

The following lemma supplements the general Lemma~\ref{lem:a.11c}. It implies in 
particular, that for unitary representations with a dense space of $\cC^1$-vectors, 
the $\cC^1$-vectors can be characterized in terms of one-parameter groups. 

\begin{lem} \mlabel{lem:2.1} 
Let $(\pi, \cH)$ be a differentiable unitary representation of the Lie group 
$G$. If $v \in \cD^1$ is such that the map 
$\omega_v \: \g \to \cH, x \mapsto \oline{\dd\pi}(x)v$ 
is continuous, then $v \in \cH^1$. 
\end{lem}

\begin{prf} Let $\gamma \: [-\eps,\eps] \to G$ be a smooth curve
with $\gamma(0) = \1$ and $\gamma'(0) = x$, and 
\[
\xi:[-\eps,\eps]\to\g\ ,\ \xi(t) := \dd\ell_{\gamma(t)^{-1}}(\gamma(t))(\gamma'(t))
\] 
be its left logarithmic derivative. 
For $w \in \cH^1$ we then obtain the relation 
\[ \frac{d}{dt} \pi(\gamma(t)^{-1})w = -\oline{\dd\pi}(\xi(t)) \pi(\gamma(t)^{-1}) w \] 
and hence 
\begin{equation}
\mlabel{dalphaw}
\frac{d}{dt} \la \pi(\gamma(t))v, w\ra  
= -\la v, \oline{\dd\pi}(\xi(t)) \pi(\gamma(t)^{-1}) w \ra 
= \la \pi(\gamma(t))\oline{\dd\pi}(\xi(t)) v, w \ra.
\end{equation}

Set $\alpha(t) := \pi(\gamma(t))v$ and 
$\alpha_w(t) := \la \alpha(t),w\ra$. 
The curve $\beta(t) := \pi(\gamma(t))\oline{\dd\pi}(\xi(t))v$ 
is also continuous because $\omega_v$ is continuous, and the 
action of $G$ on $\cH$ defined by $\pi$ is continuous. 
By \eqref{dalphaw}, for each $w \in \cH^1$ the function 
$\alpha_w$ is differentiable and 
$\alpha_w'(t) = \beta_w(t) := \la \beta(t), w \ra$. Since  $\beta_w$ is continuous, it follows that 
\[ \la \int_0^t \beta(\tau)\, d\tau, w \ra 
=  \int_0^t \beta_w(\tau)\, d\tau = \alpha_w(t) - \alpha_w(0) 
=  \la \alpha(t) - \alpha(0), w\ra.\] 
Since $\cH^1$ is assumed to be dense in $\cH$, we obtain 
$\alpha(t) = \alpha(0) + \int_0^t \beta(\tau)\, d\tau.$ 
Now the continuity of $\beta$ shows that $\alpha$ is $\cC^1$ with 
$\alpha'(0)= \beta(0) = \oline{\dd\pi}(x)v$. 
Finally Lemma~\ref{lem:a.11c} shows that $v$ is a $\cC^1$-vector. 
\end{prf}

Recall that $\cH^n$ denotes the space of $\cC^n$-vectors of $(\pi,\cH)$.
\begin{thm} \mlabel{thm:2.1} 
Let $(\pi, \cH)$ be a differentiable unitary representation of the 
Fr\'echet--Lie group $G$ with the Trotter property. 
Then $\cD^n=\cH^n$ for every $n\in\N\cup\{\infty\}$.
\end{thm}

\begin{prf} Let $n\in\N$ and $v \in \cD^n$. By Lemma~\ref{lem:contmulti}, the map 
\[ \omega_v^n \: \g^n \to \cH, \quad 
(x_1,\ldots, x_n) \mapsto 
\oline{\dd\pi}(x_1) \cdots \oline{\dd\pi}(x_n)v \] 
is continuous and $n$-linear. 
For $n = 1$, Lemma~\ref{lem:2.1} implies that $v$ is a $\cC^1$-vector. 
For $n>1$, the inductive argument for the proof of \cite[Lem.~3.4]{Ne10} works 
without change. For the reader's convenience we give the details. 

Let $n>1$ and $v\in\cD^n$. Recall that $\pi^v:G\to \cH$ denotes the orbit map of $v$, that is, $\pi^v(g):=\pi(g)v$.
Since $v\in \cD^1=\cH^1$, by induction hypothesis the map 
\[
T(\pi^v):T(G)\simeq G\times\g\to\cH\ ,\ (g,x)\mapsto T(\pi^v)(g,x)
\]
is well defined. It remains to prove that the latter map is $\cC^{n-1}$. Lemma \ref{lem:a.11c} implies that 
$T(\pi^v)(g,x)=\pi(g)\oline{\dd\pi}(x)v=\pi^w(g)$ 
where $w:=\oline{\dd\pi}(x)v\in\cD^{n-1}$. By induction hypothesis, $w\in \cH^{n-1}$. Thus, a 
direct calculation shows that the map $T(\pi^v)$ has directional derivatives of order $j\leq n-1$ and they are sums of terms of the form 
$
\pi(g)\omega_v^j(x_1,\ldots,x_j)w
$
for $j\leq n-1$. From Remark \ref{rem:cn-vec}(c) it follows that $T(\pi^v)$ is a $\cC^{n-1}$ map.

\end{prf}

\section{Unitary representations of Lie supergroups} 
\mlabel{sec:5}

We now apply Theorem~\ref{thm:2.1} to unitary representations 
of Lie supergroups $(G,\g)$. We begin by recalling the definition of Lie supergroups and their unitary representations. See \cite{CCTV06} and \cite{MNS11} for further details.

By a \emph{locally convex Lie superalgebra} we mean a Lie superalgebra 
$\g=\g_\eev\oplus \g_\ood$ 
over $\R$ or $\C$ with the following two properties.
\begin{itemize}
\item[(i)] $\g$ is a locally convex space and the Lie 
superbracket is continuous.  
\item[(ii)] $\g = \g_\eev \oplus \g_\ood$ is a topological direct sum, i.e., the involution 
$\sigma \: \g \to \g$ defined by $\sigma(x_0 + x_1) = x_0 - x_1$ for 
$x_0 \in \g_\eev$ and $x_1 \in \g_\ood$ is continuous. 
\end{itemize}
The parity of a homogeneous element $x\in\g$ is denoted by $p(x)\in\{\eev,\ood\}$.

\begin{definition}
\label{def-blsupergroup}
A \emph{Lie supergroup} is an ordered pair $(G,\g)$ with the following properties.
\begin{itemize}
\item[(i)] $G$ is a Lie group, modeled on a locally convex space. 
\item[(ii)] $\g$ is a locally convex Lie superalgebra over $\R$.
\item[(iii)] $\g_\eev=\L(G)$ is the Lie algebra of $G$. 
\item[(iv)] There exists a homomorphism 
$\Ad:G\to\mathrm{Aut}(\g)$ defining a smooth action 
$G\times \g \to \g $ by 
even automorphisms of $\g$. For every $x\in\g _\eev$ and 
$y\in\g$, we have $\dd\Ad_y(\1)(x)=[x,y]$ where $\Ad_y:G\to\g$ is defined by $\Ad_y(g):=\Ad(g)y$.

\end{itemize}
We refer to the homomorphism $\Ad:G\to \mathrm{Aut}(\g)$ 
of Definition \ref{def-blsupergroup}(iv) as the \emph{adjoint action} of 
$G$ on $\g$.
If $G$ is an analytic Lie group and the adjoint action of $G$ on 
$\g$ defines an analytic map $G \times \g \to \g$, then we call the Lie supergroup 
$(G,\g)$ \emph{analytic}. 
\end{definition}

\begin{definition}
\label{defi-smoothanalytic}
Let $(G,\g)$ be a Lie supergroup. A \emph{smooth unitary representation} 
of $(G,\g)$ is a triple $(\pi,\rho^\pi,\cH)$ satisfying the following properties.
\begin{enumerate}
\item[(SR1)] $(\pi,\cH)$ is a smooth unitary representation of $G$
on the $\mathbb Z_2$-graded Hilbert space $\cH$ such that, 
for every $g\in G$, the operator
$\pi(g)$ preserves the $\mathbb Z_2$-grading.
\item[(SR2)] For $\cB := \cH^\infty$, $\rho^\pi:\g\to\End_\C(\cB)$ 
is a representation of the Lie superalgebra $\g$. 
\item[(SR3)] $\rho^\pi(x)=\dd\pi(x)\big|_{\cB}$ for every $x\in\g_\eev$.
\item[(SR4)] $e^{-\frac{\pi i}{4}}\rho^\pi(x)$ 
is a symmetric operator for every $x\in\g_\ood$, i.e., 
$-i \rho^\pi(x) \subeq \rho^\pi(x)^*$. 
\item[(SR5)] Every element of the component group $G/G^\circ$ has a coset representative $g\in G$ such that $\pi(g)\rho^\pi(x)\pi(g)^{-1}=\rho^\pi(\Ad(g)x)$ for every $x\in\g_\ood$.
\end{enumerate}

If $(G,\g)$ is an analytic Lie supergroup, then an \emph{analytic unitary representation} 
of $(G,\g)$ is a triple $(\pi,\rho^\pi,\cH)$, where 
$(\pi, \cH)$ is an analytic 
unitary representation of $G$ on the $\mathbb Z_2$-graded Hilbert space $\cH$ such that, 
for every $g\in G$, the operator
$\pi(g)$ preserves the $\mathbb Z_2$-grading  and the other conditions above 
hold for $\cB := \cH^\omega$. 
\end{definition}

\begin{rem}
If $G$ is connected, then obviously (SR5) holds
trivially. This point is the main
difference
between Definition \ref{defi-smoothanalytic} above and the 
definition given in \cite[Def. 2]{CCTV06} for finite-dimensional Lie groups,
where it is assumed that 
\begin{equation}
\label{eqn-condin}
\pi(g)\rho^\pi(x)\pi(g)^{-1}=\rho^\pi(\Ad(g)x)\ \text{ for every $x\in\g_\ood$ and every $g\in G$,}  
\end{equation}
while  the infinitesimal action is supposed to satisfy a weaker condition. 
Indeed Proposition~\ref{proposition-conjugacy} below implies that for a (possibly disconnected) $G$ 
equation 
\eqref{eqn-condin} follows from Definition~\ref{defi-smoothanalytic}.  
\end{rem}

We will need a slightly more general notion than smooth and analytic unitary representations, which
we introduce in the next definition (see \cite[Def. 4.2]{MNS11}).

\begin{definition}
\label{definition-of-pseudo} Let $(G,\g)$ be a Lie supergroup. 
A \emph{pre-representation} of 
$(G,\g)$ is a 4-tuple 
$
(\,\pi,\cH,\cB,\rho^\cB\,)$
which satisfies
the following properties.
\begin{enumerate} 
\item[(PR1)] $(\pi,\cH)$ is a smooth unitary representation of $G$  on the 
$\mathbb Z_2$-graded Hilbert space $\cH=\cH_\eev\oplus\cH_\ood$. Moreover, $\pi(g)$ is an even operator for every $g\in G$.
\item[(PR2)] $\cB$ is a dense $\mathbb Z_2$-graded 
subspace of $\cH$ contained in $\cD^1 := \bigcap_{x\in \g_\eev}\cD\big(\oline{\dd\pi}(x)\big).$ 
\item[(PR3)] $\rho^\cB:\g\to \End_\mathbb C(\cB)$  
 is a representation of the Lie superalgebra $\g$.
\item[(PR4)]  If $x\in\g_\eev$ then $\rho^\cB(x)=\oline{\dd\pi}(x)\big|_\cB$ and 
$\rho^\cB(x)$ is essentially skew-adjoint.
\item[(PR5)] If $x\in \g_\ood$ then $e^{-\frac{\pi i}{4}}\rho^\cB(x)$ is a 
symmetric operator, i.e., 
$-i \rho^\cB(x) \subeq \rho^\cB(x)^*$. 

\item[(PR6)] For every element of the component group $G/G^\circ$, there exists a coset representative
$g\in G$ such that $\pi(g)^{-1}\cB\subeq \cB$ and 
\[
\pi(g)\rho^\cB(x)\pi(g)^{-1}=\rho^\cB(\Ad(g)x)\text{ for every $x\in\g_\ood$.}
\]
\end{enumerate}

\end{definition}
In the following we write
\[ \cD^\infty := \bigcap_{n\in \N}\cD^n 
\quad \mbox{ and }\quad 
\cD^n:=\bigcap_{x_1,\ldots,x_n\in\g_\eev}\cD\big(\oline{\dd\pi}(x_1)\cdots\oline{\dd\pi}(x_n)\big).\] 
The set of smooth (resp., analytic) vectors of the unitary representation $(\pi,\cH)$ is denoted by $\cH^\infty$ (resp., $\cH^\omega$).

\begin{rem}
(i) Observe that in (PR3) there are no continuity assumptions on the  
map~$\rho^\cB$. 

(ii) (PR2/3) imply that $\cB\subeq \cD^\infty$. 
In addition, Theorem~\ref{thm:2.1} asserts that, if $G$ is a Fr\'echet--Lie group with the Trotter property, 
then $\cH^\infty = \cD^\infty$, so that we obtain $\cB \subeq \cH^\infty$. 
\end{rem}

The following lemma is stated in \cite[Lemma~4.4]{MNS11} for 
Banach--Lie supergroups, but is remains true with the same proof in the general context.

\begin{lemma}
\label{lem-extension-of-operators}
Let $(G,\g)$ be a Lie supergroup and
$\big(\pi,\cH,\cB,\rho^\cB\big)$
be a pre-representation of $(G,\g)$. Then the following assertions hold.
\begin{enumerate}
\item[\rm(i)] For every $x\in \g_\eev$, we have $\overline{\rho^\cB(x)}=\oline{\dd\pi}(x)$. In particular
$\cD^\infty\subeq \cD^1 \subeq 
\cD(\overline{\rho^\cB(x)})$. 
\item[\rm(ii)] For every $x\in \g_\ood$, the operator $e^{-\frac{\pi i}{4}}\rho^\cB(x)$ is essentially 
self-adjoint and $\overline{\rho^\cB(x)}^2=\frac{1}{2}\oline{\dd\pi}([x,x])$. In particular 
$\cD^\infty\subeq \cD^1 \subeq \cD(\overline{\rho^\cB(x)})$. 
\end{enumerate}
\end{lemma}

\begin{defn}
Let $(\pi,\cH,\cB,\rho^\cB)$ be a pre-representation of a Lie supergroup 
$(G,\g)$. 
For every  $x=x_\eev+x_\ood\in\g^\C$  we
define a linear operator $\tilde\rho^\cB(x)$ on $\cH$ 
with  $\cD(\tilde\rho^\cB(x))=\cD^\infty$ as follows. If $x_\eev=a_\eev+ib_\eev$ and
$x_\ood=a_\ood+ib_\ood$, where $a_\eev,b_\eev\in \g_\eev$ and $a_\ood,b_\ood\in\g_\ood$, 
then we put 
\[
\tilde\rho^\cB(x)v=\overline{\rho^\cB(a_\eev)}v+
i\overline{\rho^\cB(b_\eev)}v
+\overline{\rho^\cB(a_\ood)}v
+i\overline{\rho^\cB(b_\ood)}v 
\quad \mbox{ for } \quad  v\in\cD^\infty.\] 
\end{defn}

The following auxiliary lemma is \cite[Lemma~2.5]{MNS11}. 

\begin{lemma}
\label{lem-two-oper}
Let $P_1$ and $P_2$ be two symmetric operators on a complex Hilbert space $\cH$ such that 
$\mathcal D(P_1)=\mathcal D(P_2)$. Let $\mathscr L\sseq\mathcal D(P_1)$ 
be a dense linear subspace of $\cH$ such that 
$P_1\big|_{\mathscr L}=P_2\big|_{\mathscr L}$. Assume that the latter operator is essentially self-adjoint.
Then $P_1=P_2$.
\end{lemma}

\begin{proposition} \label{smooth-preserve}
Let $(G,\g)$ be a Lie supergroup and
$\big(\pi,\cH,\cB,\rho^\cB\big)$
be a pre-representation of $(G,\g)$. 
Then the following assertions hold.
\begin{description}
  \item[\rm(i)] 
$\tilde\rho^\cB(x)\cD^\infty \subeq \cD^\infty$ for every $x\in\g^\C$. 
  \item[\rm(ii)] 
The map $\tilde\rho^\cB \: \g^\C \to \End(\cD^\infty)$ is a homomorphism 
of Lie superalgebras, i.e., it is linear, and, 
if $x,y\in\g^\C$ are homogeneous, then 
\[
\tilde\rho^\cB([x,y])=\tilde\rho^\cB(x)\tilde\rho^\cB(y)
-(-1)^{p(x)p(y)}\tilde\rho^\cB(y)\tilde\rho^\cB(x). \]
\item[\rm(iii)] For each $v \in \cH^\infty$, the map 
$\omega_v \: \g \to \cH, x \mapsto \tilde\rho^\cB(x)v$ is linear and continuous.
\end{description} 
\end{proposition}

\begin{proof} (i) By Lemma \ref{lem-extension-of-operators}(i) and the definition of $\tilde\rho^\cB$  
it suffices to prove the statement when $x\in \g_\ood$. 
Therefore it is enough to prove that  
\begin{equation}
\label{inclusion-eq}
\overline{\rho^\cB(x)}v\in\cD^n
\mbox{ for every } x\in\g_\ood, v\in\cD^{n+1}, n \in \N. \end{equation}
Let $y\in\g_\eev$. For every $w\in\cB$ and $v \in \cD^2$, 
using Lemma~\ref{lem-extension-of-operators}, 
we can write
\begin{align*}
&\langle \overline{\rho^\cB(x)}v,\oline{\dd\pi}(y)w\rangle=
\langle\overline{\rho^\cB(x)}v,\rho^\cB(y)w\rangle 
=e^\frac{\pi i}{2}\langle v,\overline{\rho^\cB(x)}\rho^\cB(y)w\rangle\\
&=e^\frac{\pi i}{2}\langle v,\rho^\cB(y)\rho^\cB(x)w+
\rho^\cB([x,y])w\rangle
=e^\frac{\pi i}{2}
\langle v,\rho^\cB(y)\rho^\cB(x)w\rangle+
e^\frac{\pi i}{2}\langle v,\rho^\cB([x,y])w\rangle  \\
&=\langle\overline{\rho^\cB(x)}\oline{\dd\pi}(y)v,w\rangle+
\langle \overline{\rho^\cB([x,y])}v,w\rangle.
\end{align*}
It follows that the complex linear functional
$\cB\to\C, w\mapsto \langle \overline{\rho^\cB(x)}v,\oline{\dd\pi}(y)w\rangle$ 
is continuous, i.e., 
$\overline{\rho^\cB(x)}v\in\cD\big((\oline{\dd\pi}(y)\big|_\cB)^*\big)$. 
Since $\oline{\dd\pi}(y)\big|_\cB=\rho^\cB(y)$ is essentially skew-adjoint,
from Lemma
\ref{lem-extension-of-operators}(i) it follows that $(\oline{\dd\pi}(y)\big|_\cB)^*=-\oline{\dd\pi}(y)$, i.e., 
$\overline{\rho^\cB(x)}v\in\cD(\oline{\dd\pi}(y))$. 
This proves \eqref{inclusion-eq} for $n=1$.

For $n>1$ the proof of \eqref{inclusion-eq} can be completed by induction. 
Let $x_1,\ldots,x_n\in\g_\eev$ and 
$v\in\mathscr \cD^{n+1}$. 
Using the induction hypothesis, for every $w\in\cB$ we can write
\begin{align*}
\langle&\oline{\dd\pi}(x_{n-1})\cdots\oline{\dd\pi}(x_1)\overline{\rho^\cB(x)}v,\oline{\dd\pi}(x_n)w\rangle
=e^{(n-\frac{1}{2})\pi i}
\langle v,\rho^\cB(x)\rho^\cB(x_1)\cdots\rho^\cB(x_n)w\rangle\\
&=e^{(n-\frac{1}{2})\pi i}\langle v,
\rho^\cB([x,x_1])\rho^\cB(x_2)\cdots\rho^\cB(x_n)w+
\rho^\cB(x_1)\rho^\cB(x)\rho^\cB(x_2)\cdots\rho^\cB(x_n)w \rangle\\
&=\langle\oline{\dd\pi}(x_n)\cdots\oline{\dd\pi}(x_2)\overline{\rho^\cB([x,x_1])}v,w\rangle
+\langle\oline{\dd\pi}(x_n)\cdots\oline{\dd\pi}(x_2)\overline{\rho^\cB(x)}\oline{\dd\pi}(x_1)v,w\rangle.
\end{align*}
An argument similar to the case $n=1$ proves that 
\[
\oline{\dd\pi}(x_{n-1})\cdots\oline{\dd\pi}(x_1)\overline{\rho^\cB(x)}v\in\cD(\oline{\dd\pi}(x_n)).
\]
Consequently, $\overline{\rho^\cB(x)}v\in\cD^n$.

(ii) (cf.\ \cite[Prop.~4.6]{MNS11})  
First we show that $\tilde\rho^\cB$ is linear. 
By Lemma \ref{lem-extension-of-operators}(i) and the definition of $\tilde\rho^\cB$, 
it is enough to prove that, for every $v\in\cD^\infty$, the map
\[ \g_\oo \to\cH,\quad   x\mapsto\overline{\rho^\mathscr B(x)}v\]
is $\R$-linear. 
Let $x\in\g_\oo$ and $a\in \R$. Then the equality
\begin{equation}
\label{eqn-scalara}
\overline{\rho^\cB(ax)}v=a\overline{\rho^\cB(x)}v
\end{equation}
holds for every $v\in\cB$, and therefore, by Lemma \ref{lem-two-oper},
 it also holds for every $v\in\cD^\infty$.
A similar reasoning proves that, if $x,y\in \g_\oo$, then, for every $v\in\cD^\infty$,  we have
\begin{equation}
\label{eqn-sumxy}
\overline{\rho^\cB(x+y)}v=\overline{\rho^\cB(x)}v+\overline{\rho^\cB(y)}v.
\end{equation} 

It suffices to prove the commutation relation for $x,y\in\g$. 
Depending on the parities of $x$ an $y$, there are four cases to consider, 
but the argument for all of them is essentially the same.
For example, if $x\in\g_\eev$ and $y\in\g_\ood$, then we define two operators $P_1$ and $P_2$ with domains
$\mathcal D(P_1)=\mathcal D(P_2)=\cD^\infty$  as follows. For $v\in\cD^\infty$  we set
\[
P_1v=e^{-\frac{\pi i}{4}}\overline{\rho^\cB([x,y])}v\ \text{ and }\ 
P_2v=e^{-\frac{\pi i}{4}}\Big(\overline{\rho^\cB(x)}\,\overline{\rho^\cB(y)}v-
\overline{\rho^\cB(y)}\,\overline{\rho^\cB(x)}v\Big).
\]
Then $P_1$ and $P_2$ are both symmetric, $P_1\big|_{\cB}=P_2\big|_{\cB}$, and by
Lemma \ref{lem-extension-of-operators}(ii),  
the operator $P_1\big|_{\cB}$ is 
essentially self-adjoint. Lemma \ref{lem-two-oper} implies that $P_1=P_2$.

(iii) The linearity of $\omega_v$ follows from (ii). The continuity 
of $\omega_v$ on $\g_{\eev}$ follows from the definition of a smooth vector. 
Therefore it remains to show that $\omega_v\res_{\g_\ood}$ is continuous in $0$. 
This can be reduced to continuity of $\omega_v$ on $\g_{\eev}$ using 
the estimate 
\begin{align}
\label{eq:estim}
 \|\oline{\rho^\cB(y)}v\|&=
 \langle
\oline{\rho^\cB(y)}v,
\oline{\rho^\cB(y)}v 
\rangle^\frac{1}{2}\\
&=
| \langle
v,
\oline{\rho^\cB(y)}^2v 
\rangle|^\frac{1}{2}
 \leq \frac{1}{\sqrt 2} \|v\|^{1/2} \|\oline{\dd\pi}([y,y])v\|^{1/2}, 
\quad y \in \g_\ood, v \in \cD^2 
\notag\end{align} 
which is a consequence of 
the Cauchy--Schwarz inequality.
\end{proof}

\begin{lem} 
\mlabel{rem:1.9}
Let $(G,\g)$ be a Fr\'echet--Lie 
supergroup such that $G$ has the Trotter property.
Let $(\pi,\cH,\cB,\rho^\cB)$ be a pre-representation of $G$ and $v\in\cH^\infty$. 
For every $x\in\g$  
the map 
\begin{equation}
\label{secondmapi}
G\times\g_\eev \to \cH, \quad (g,y) \mapsto \oline{\rho^\cB(x)}\oline{\dd\pi}(y)\pi(g)v 
\end{equation}
 is continuous.

\end{lem}

\begin{prf}

Since $G$ has the Trotter property,  
by
Theorem \ref{thm:2.1} and Proposition \ref{smooth-preserve}(ii) we have $\tilde\rho^\cB(z)\cH^\infty\subseteq \cH^\infty$ for every $z\in\g^\C$.

We can assume $x$ is homogeneous. First assume $x\in\g_\eev$. By Remark \ref{rem:cn-vec}(c),
the maps
\[
G\times\g_\eev\to\cH\ ,\ 
(g,x_1)\mapsto \pi(g)\oline{\dd\pi}(x_1)v
\]
and
\[
G\times\g_\eev\times \g_\eev\to\cH\ ,\ 
(g,x_1,x_2)\mapsto \pi(g)\oline{\dd\pi}(x_1)\oline{\dd\pi}(x_2)v
\]
are continuous. Therefore continuity of \eqref{secondmapi}
follows from
\[
\oline{\dd\pi}(x)\oline{\dd\pi}(y)
\pi(g)v=\pi(g)
\oline{\dd\pi}(\Ad(g^{-1})x)
\oline{\dd\pi}(\Ad(g^{-1})y)v
\]
and smoothness of the map $G\to\g_\eev\ ,\ g\mapsto\Ad(g)z$ for every $z\in\g_\eev$.

Next assume $x\in\g_\ood$. From the preceding argument it follows that the maps
\begin{equation}\mlabel{sevvom:eq}
G\times \g_\eev\to\cH\ ,\ (g,x_1)\mapsto\oline{\dd\pi}([x,x])\oline{\dd\pi}(x_1)
\pi(g)v
\end{equation}
and 
\begin{equation}
\mlabel{sevvom:eqg}
G\times \g_\eev\to\cH\ ,\ (g,x_1)\mapsto\oline{\dd\pi}(x_1)\pi(g)v
\end{equation}
are continuous. From \eqref{eq:estim} it follows that
\begin{align*}
\|\oline{\rho^\cB(x)}&\big(\oline{\dd\pi}(y_1)\pi(g_1)v-\oline{\dd\pi}(y_2)\pi(g_2)v\big)\|\\
&\leq
\frac{1}{\sqrt{2}}\|\oline{\dd\pi}(y_1)\pi(g_1)v-\oline{\dd\pi}(y_2)\pi(g_2)v)\|^\frac{1}{2}
\cdot
\|\oline{\dd\pi}([x,x])
(\oline{\dd\pi}(y_1)\pi(g_1)v-\oline{\dd\pi}(y_2)\pi(g_2)v)
\|^\frac{1}{2}
\end{align*}
for $g_1,g_2\in G$ and $y_1,y_2\in\g_\eev$. Consequently, continuity of \eqref{secondmapi}
follows from continuity of \eqref{sevvom:eq}
and \eqref{sevvom:eqg}.
\end{prf}

\begin{rem}
\label{rem:6.11}
Let $\eps>0$ and $\alpha:(-\eps,\eps)\to G$ be a smooth curve. The right logarithmic derivative of $\alpha$ is the smooth curve
$\rlg(\alpha):(-\eps,\eps)\to\g$ 
defined by
\[
\rlg(\alpha)_s:=\dd\varrho_{\alpha(s)^{-1}}(\alpha(s))(\alpha'(s))
\]
where $\varrho_g:G\to G$ is the right multiplication by $g\in G$.
The product rule for $\rlg$ is given by
\[
\rlg(\alpha\beta)_s=\rlg(\alpha)_s+\Ad(\alpha(s))\rlg(\beta)_s.
\]

If $(\pi,\cH)$  is a unitary representation of $G$ and 
$v\in\cH^\infty$ then $\pi(\alpha(s))v\in\cH^\infty$ and therefore
\begin{align*}
\lim_{h\to 0}\frac{1}{h}&\big(\pi(\alpha(s+h))v-\pi(\alpha(s))v\big)\\
&=
\lim_{h\to 0}\frac{1}{h}\big(\pi(\alpha(s+h)\alpha(s)^{-1})\pi(\alpha(s))v-\pi(\alpha(s))v\big)
=
\oline{\dd\pi}(\rlg(\alpha)_s)\pi(\alpha(s))v.
\end{align*}
Similarly, we obtain with the left logarithmic derivative $\delta(\alpha)$: 
\begin{align*}
\lim_{h\to 0}\frac{1}{h}&\big(\pi(\alpha(s+h))v-\pi(\alpha(s))v\big)\\
&= \pi(\alpha(s))
\lim_{h\to 0}\frac{1}{h}\big(\pi(\alpha(s)^{-1}\alpha(s+h))v-v\big)
=\pi(s) \oline{\dd\pi}(\delta(\alpha)_s)v.
\end{align*}

\end{rem}

Our next goal is to prove 
Proposition \ref{proposition-conjugacy} below. The proof of this proposition is based on a variation of
a subtle lemma from \cite[Chap.~3]{JM84} (see \cite{Me11} and \cite{MNS11} as well).

\begin{lem}

\mlabel{lem-jorgenson}
Let $(\pi,\cH)$ be a smooth unitary representation of a Lie group $G$. 
Let $B$ be a closable operator on $\cH$ such that $\cH^\infty\subseteq\cD(B)$ and $B\cH^\infty\subseteq \cH^\infty$.
Let $\alpha:(-\eps,\eps)\to G$ be a smooth curve. Let $v\in\cH^\infty$ and
set 
\[
\eta:(-\eps,\eps)\to \cH\ ,\ \eta(s):=B\pi(\alpha(s))v.
\]
Assume that the map
\[
\beta:(-\eps,\eps)\to \cH\ ,\ 
\beta(s):=B\oline{\dd\pi}(\rlg(\alpha)_s)\pi(\alpha(s))v
\]
is continuous. Then $\eta$ is differentiable and $\displaystyle\frac{d}{ds}\eta(s)=\beta(s)$ for every  $s\in(-\eps,\eps)$.
\end{lem}

\begin{prf} Set $\gamma(s):=\pi(\alpha(s))v$. Since $v$ is a smooth vector 
and $\alpha$ a smooth curve, $\gamma$ is also smooth. 
Remark \ref{rem:6.11} implies that 
$\gamma'(s)=\oline{\dd\pi}(\rlg(\alpha)_s)\pi(\alpha(s))v$, 
and therefore 
\begin{equation}
\mlabel{eq:intinhb}
\pi(\alpha(s))v-\pi(\alpha(0))v=\gamma(s)-\gamma(0)=\int_0^s
\big(\oline{\dd\pi}(\rlg(\alpha)_s)\pi(\alpha(s))v\big) dt.
\end{equation}
Let $\cH_B$ denote the completion of $\cH^\infty$ with respect to the norm 
$\|w\|_B:=\|w\|+\|Bw\|$. Then $\cH_B$ is a Banach space and $B:\cH_B\to \cH$ is a continuous linear map. Continuity of $\beta$ implies that the map $\gamma':(-\eps,\eps)\to\cH_B$ is continuous, and therefore 
\eqref{eq:intinhb} holds in $\cH_B$. Since $B:\cH\to\cH_B$ is linear and continuous, we have
\begin{equation*}
\eta(s)-\eta(0)=B\pi(\alpha(s))v-B\pi(\alpha(0))v=B(\gamma(s)-\gamma(0))=\int_0^s
\left(B\oline{\dd\pi}(\rlg(\alpha)_s)\pi(\alpha(s))v\right) dt.
\end{equation*}
The last equality immediately implies that
$\displaystyle\frac{d}{ds}\eta(s)=\beta(s)$ for every $s\in(-\eps,\eps)$.
\end{prf}

\begin{proposition}
\label{proposition-conjugacy}
Let $(G,\g)$ be a Fr\'echet--Lie supergroup. Assume that $G$ is connected and has the Trotter property. Let
$\big(\pi,\cH,\cB,\rho^\cB\big)$
be a pre-representation of $(G,\g)$. Then for every $g\in G$, 
every $x\in\g_\ood^\C$, and every 
$v\in\cH^\infty$ we 
have 
\begin{equation}
\label{eqn-conjugacy}
\pi(g)\tilde\rho^\cB(x)\pi(g)^{-1}v=\tilde\rho^\cB(\Ad(g)x)v.
\end{equation}
\end{proposition}

\begin{proof}
By
Theorem \ref{thm:2.1} and Proposition \ref{smooth-preserve}(ii) we have $\tilde\rho^\cB(z)\cH^\infty\subseteq \cH^\infty$ for every $z\in\g^\C$.
Let $\alpha:\R\to G$ be a smooth curve such that $\alpha(0)=\1$ and $\alpha(1)=g$.
Set $A(s):=\pi(g\alpha(s)^{-1})$ and $\gamma(s):=\Ad(\alpha(s))x$ for  $s\in\R$.
For $v\in\cH^\infty$ and $s\in\R$ define 
\[
K(s):\g_\ood^\C\to \cH\,,\, K(s)x=A(s)
\tilde\rho^\cB(x)
A(s)^{-1}v.
\]
The left hand side of \eqref{eqn-conjugacy}  is equal to $K(0)\gamma(0)$ and the right hand side 
of \eqref{eqn-conjugacy} is equal to $K(1)\gamma(1)$. Therefore it suffices to prove that 
the map
$s\mapsto K(s)\gamma(s)$ is constant. To this end, we will prove that 
it is differentiable and that its derivative is identically zero.

By Proposition~\ref{smooth-preserve}(iii), for every $t\in\R$, the linear map 
$K(t)$ is continuous. From Lemma~\ref{rem:1.9} we derive that 
for every $y\in\g^\C_\ood$, the curve 
\[
\R\to\cH\,,\, s\mapsto \tilde\rho^\cB(y)
\oline{\dd\pi}(\rlg(\alpha)_s)A(s)^{-1}v
\] 
is continuous. It follows from Lemma 
\ref{lem-extension-of-operators}(ii) and 
Lemma \ref{lem-jorgenson} that the curve 
\[
\eta_y:\R\mapsto\cH\, ,\, \eta(s)
=\tilde\rho^\cB(y)A(s)^{-1}v
=\tilde\rho^\cB(y)\pi(\alpha(s))\pi(g)^{-1}v
\]
is differentiable, and $\eta'_y(s)=\tilde\rho^\cB(y)
\oline{\dd\pi}(\rlg(\alpha)_s)A(s)^{-1}v=
\tilde\rho^\cB(y)
\tilde\rho^\cB(\rlg(\alpha)_s)A(s)^{-1}v
$. 

Next we show that, for every $y\in\g^\C_\ood$, the map $K^y \: \R\to\cH, s\mapsto K(s)y$ 
is differentiable and we compute its derivative. 
Observe that
\begin{align*}
(K^y)'(s)&=\frac{d}{ds}(A(s)\eta_y(s))=\lim_{h\to 0}\frac{1}{h}
\big(A(s+h)\eta_y(s+h)-A(s)\eta_y(s)\big)
\end{align*}
and 
\begin{align}
\label{eqn-as+h}
\notag
\frac{1}{h}&\big(A(s+h)\eta_y(s+h)-A(s)\eta_y(s)\big)\\
&=A(s+h)\Big(\frac{1}{h}\big(\eta_y(s+h)-\eta_y(s)\big)\Big)+
\frac{1}{h}\big(A(s+h)\eta_y(s)-A(s)\eta_y(s)\big).
\end{align}
Since the map $\R \times \cH \to \cH, (s,v) \mapsto A(s)v$ is continuous, 
when $h\to 0$ we obtain 
\[
A(s+h)\Big(\frac{1}{h}\big(\eta_y(s+h)-\eta_y(s)\big)\Big)\to 
A(s)\eta_y'(s).
\]
Since 
$\eta_y(s)\in\cH^\infty$, from Remark 
\ref{rem:6.11} and the relation
$\delta(\alpha^{-1})_s=-\rlg(\alpha)_s$ 
it follows that, 
as $h\to 0$, the second term in \eqref{eqn-as+h} converges to 
\[  -A(s)\tilde\rho^\cB(\rlg(\alpha)_s)\eta_y(s) = 
 -A(s)\tilde\rho^\cB(\rlg(\alpha)_s)\tilde\rho^\cB(y)A(s)^{-1}v. \]
By Proposition \ref{smooth-preserve}(ii) we conclude that
\begin{equation}
\label{derivofksx}
(K^y)'(s) =A(s)\eta'_y(s)-A(s)\tilde\rho^\cB(\rlg(\alpha)_s)\tilde\rho^\cB(y)A(s)^{-1}v=A(s)
[\tilde\rho^\cB(y),\tilde\rho^\cB(\rlg(\alpha)_s)]A(s)^{-1}v.
\end{equation}
Finally we prove that $\frac{d}{ds}(K(s)\gamma(s))=0$ for every $s\in\R$.
We have
\begin{align}
\label{eqn-ks+hfinalapprox}
\notag \frac{1}{h}&\big(K(s+h)\gamma(s+h)-K(s)\gamma(s)\big)\\
&=K(s+h)\Big(\frac{1}{h}\big(\gamma(s+h)-\gamma(s)\big)\Big)+
\frac{1}{h}\big(K(s+h)\gamma(s)-K(s)\gamma(s)\big).
\end{align}
The first term in \eqref{eqn-ks+hfinalapprox} can be written as 
\[
K(s+h)\Big(\frac{1}{h}\big(\gamma(s+h)-\gamma(s)\big)-\gamma'(s)\Big)+K(s+h)\gamma'(s).
\]
Differentiability of the map $K^y$ implies that it is continuous. Thus, if 
$I$ is a compact interval containing $s$, 
then $\sup_{t\in I}\|K(t)y\|<\infty$. Since $\g$ is Fr\'echet, the Banach--Steinhaus Theorem
implies that $(K(t))_{t \in I}$ is equicontinuous. 
Since $\gamma'(s)=[\rlg(\alpha)_s,\Ad(\alpha(s))x]$,
it follows that, for $h\to 0$, the first term in 
\eqref{eqn-ks+hfinalapprox} converges to 
\[
K(s)\gamma'(s)=A(s)\tilde\rho^\cB([\rlg(\alpha)_s,\Ad(\alpha(s))x])
A(s)^{-1}v 
= A(s)\tilde\rho^\cB([\rlg(\alpha)_s,\gamma(s)])A(s)^{-1}v.
\]
By \eqref{derivofksx}, the second term in \eqref{eqn-ks+hfinalapprox} converges to
\[
A(s)[\tilde\rho^\cB(\gamma(s)),\tilde\rho^\cB(\rlg(\alpha)_s)]A(s)^{-1}v 
= A(s)\tilde\rho^\cB([\gamma(s),\rlg(\alpha)_s])A(s)^{-1}v.
\]
We conclude that $\frac{d}{ds}(K(s)\gamma(s))=0$.
\end{proof}

We can now prove the following theorem, which asserts 
that every pre-representation of a Lie supergroup 
corresponds to a unique unitary representation.

\begin{theorem}
\label{thm-stabili}
\textbf{\upshape(Stability Theorem)}
Let  $(\pi,\cH,\cB,\rho^\cB)$ be a  pre-representation 
of the Fr\'echet--Lie supergroup $(G,\g)$ and assume that 
$G$ has the Trotter property. 

{\rm(a)} There exists a unique linear map 
\[\rho^\pi:\g\to\End_\C(\cH^\infty)\]
such that 
$\rho^\pi(x)\big|_{\cB}=\rho^\cB(x)$ and
$(\pi,\rho^\pi,\cH)$ is a smooth unitary representation of $(G,\g)$. 

{\rm(b)} If the representation $(\pi,\cH)$ of $G$ is analytic, then  there exists a unique map 
\[\rho^\pi:\g\to\End_\C(\cH^\omega)\] 
such that 
$\overline{\rho^\pi(x)}\big|_{\cB}=\rho^\cB(x)$ and
$(\pi,\rho^\pi,\cH)$ is an analytic unitary representation of $(G,\g)$. 
\end{theorem}

\begin{proof} First we recall from Theorem~\ref{thm:2.1} 
that the Trotter property of $G$ and the smoothness of the representation 
imply that $\cD^\infty = \cH^\infty$. 

(a)  To prove the existence of $\rho^\pi$, we set $\rho^\pi(x)=\tilde\rho^\cB(x)$ for every $x\in\g$. Proposition \ref{smooth-preserve} implies that 
$\tilde\rho^\cB(x)\in\End_\C(\cH^\infty)$ and 
$\tilde\rho^\cB$ is a representation of the Lie superalgebra~$\g$. 
To prove the conjugacy invariance relation 
of (SR5), for every element of $G/G^\circ$ we take a coset representative
$g\in G$ which satisfies the condition (PR6) and apply Lemma~\ref{lem-two-oper} with 
$P_1=e^{-\frac{\pi i}{4}}\pi(g)\tilde\rho^\cB(x)\pi(g)^{-1}$, 
$P_2=e^{-\frac{\pi i}{4}}\tilde\rho^\cB(\Ad(g)x)$, and $\mathscr L=\cB$.

To prove uniqueness, it suffices to show that, if $(\pi,\rho^\pi,\cH)$ is a 
smooth unitary representation such that for every $x\in\g$ we have
$\rho^\pi(x)\big|_{\cB}=\rho^\cB(x)$ 
then for every $x\in\g$ we have 
\begin{equation}
\label{eqn-rhopibigi}
\rho^\pi(x)\big|_{\cH^\infty}=\tilde\rho^\cB(x).
\end{equation}
It suffices to prove \eqref{eqn-rhopibigi} when $x$ is 
homogeneous. If $x\in\g_\ood$ then 
by Lemma \ref{lem-extension-of-operators}(ii) the operator 
$e^{-\frac{\pi i}{4}}\tilde\rho^\cB(x)\big|_\cB$ is essentially self-adjoint. Therefore
\eqref{eqn-rhopibigi} follows from setting 
$P_1=e^{-\frac{\pi i}{4}}\rho^\pi(x)\big|_{\cH^\infty}$, $P_2=e^{-\frac{\pi i}{4}}\tilde\rho^\cB(x)$ and 
$\mathscr L=\cB$ in Lemma~\ref{lem-two-oper}. The argument for $x\in\g_\eev$ is similar.

(b) The proof of uniqueness is similar to the one given for the smooth case. 
For the existence,  it remains to verify that if $(\pi,\rho^\pi,\cH)$
is the smooth representation obtained in (a), then $\rho^\pi(y)\cH^\omega\subseteq\cH^\omega$ for every $y\in\g_\ood$.
Let $y \in \g_\ood$, $v \in \cH^\omega$ and 
$w := \oline{\rho^\cB(y)}v$. We have to show that 
$w\in \cH^\omega$, and, in view of \cite[Thm.~5.2]{Ne11}, it suffices to verify that 
the function 
\begin{align*}
G \to \C\ ,\ 
g \mapsto \la \pi(g)w,w\ra 
\end{align*}
is analytic. Note that
\begin{align*}
\la \pi(g)w,w\ra
= \la \pi(g)\oline{\rho^\cB(y)}v, w \ra 
= \la \oline{\rho^\cB(\Ad(g)y)}\pi(g)v, w \ra 
= -i \la \pi(g)v, \oline{\rho^\cB(\Ad(g)y)}w \ra.
\end{align*}
The orbit map of $v$ is analytic, and 
by Proposition~\ref{smooth-preserve}(iii)
the linear map 
$\g \to \cH, z \mapsto \oline{\rho^\cB(z)}w$ 
is continuous. The assertion follows from
analyticity of the map $G \to \g\,,\, g \mapsto \Ad(g)y$. 
\end{proof}

{\bf Acknowledgments} We thank D.\ Beltita and H. Gl\"ockner for 
various remarks on earlier versions of this paper.

\appendix

\section{Some results on Lie groups of maps} 

The results of this appendix will be used in Appendix 
\ref{app:b}. Let $E$ be a locally convex space, $M$ be a 
smooth finite dimensional manifold (possibly with boundary) and $K$ be a Lie group (possibly infinite dimensional) with Lie algebra $\mathfrak k$. 
In the following we write $\Omega^1_{\cC^r}(M,E)$ for the space 
of $E$-valued $1$-forms on $M$ defining $\cC^r$-functions 
$TM \to E$. The space of $E$-valued smooth forms will be denoted by $\Omega^1(M,E)$. 

We endow $\Omega^1_{\cC^r}(M,E)$  
with the topology induced by the embedding 
\[
\Omega^1_{\cC^r}(M,E)\into\cC^r(TM,E),
\] 
where $TM$ is the tangent bundle and 
 $\cC^r(TM,E)$ is endowed with the compact open $\cC^r$-topology,
so that $\Omega^1_{\cC^r}(M,E)$ is a closed subspace of 
$\cC^r(TM,E)$. The space
$\Omega^1(M,E)$ is endowed with the topology induced by
the diagonal embedding 
\[
\Omega^1(M,E)\into\prod_{r=1}^\infty \Omega^1_{\cC^r}(M,E).
\]

\begin{lem} \mlabel{lem:smoothprop} 
Let $M$ be a compact smooth manifold (possibly with boundary) 
and $K$ a Lie group with Lie algebra~$\fk$. 
Then, for each $r \in \N_0 \cup \{\infty\}$, the action of the Lie group 
$\cC^r(M,K)$ on $\Omega^1_{\cC^r}(M,\fk)$ by 
$(g,\alpha) \mapsto \Ad(g)\alpha$ is smooth. 
\end{lem}

\begin{prf} Assume that $d = \dim M$. 
Every covering $(U_i)_{i \in I}$ of $M$ by compact submanifolds with 
boundary, which are 
diffeomorphic to $d$-dimensional balls and whose interiors define an 
atlas, yields an embedding 
$$ \Omega^1_{\cC^r}(M,\fk) 
\into \prod_{i \in I} \Omega^1_{\cC^r}(U_i,\fk) 
\cong  \prod_{i \in I} \cC^r(U_i,\fk)^d. $$
Therefore it suffices to show that the action of $\cC^r(M,K)$ on each space $\cC^r(U_i,\fk)$, given by 
$$ (g,f) \mapsto \Ad(g)f = \sigma_{\Ad} \circ (g,f) $$
is smooth. This action factors through the Lie group morphisms 
$$ \cC^r(M,K) \to  \cC^r(U_i,K), $$ 
and for the Lie groups $G := \cC^r(U_i,K)$, it coincides with the adjoint action 
on $\L(G) \cong \cC^r(U_i,\fk)$, which is smooth. 
This proves the lemma. 
\end{prf} 

\begin{lem} \mlabel{lem:a.2} (\cite{GN12}) \mlabel{lem:smooth-compo} Let $E_1$ 
and $E_2$ be locally convex spaces, $U_1 \subeq E_1$ open, 
$M$ a compact smooth manifold (possibly with boundary),  and 
$\phi \: U_1 \to E_2$ be a smooth map. Then, for each $r \in 
\N \cup \{\infty\}$, the map 
$$ \phi_* \: \cC^r(M,U_1) \to \cC^r(M,E_2), \quad f \mapsto \phi \circ f $$
is smooth. 
\end{lem}

\begin{lem} \mlabel{lem:smoothprop2} 
Let $U$ be an open subset of a locally convex space $E$, 
$F$ a locally convex space, 
$M$ a compact manifold (possibly with boundary), $r \in \N \cup \{ \infty\}$, and 
$\alpha \in \Omega^1(U,F)$. Then the map 
$$ \cC^r(M,U) \to \Omega^1_{\cC^{r-1}}(M,F), \quad 
f \mapsto f^*\alpha $$
is smooth if $\cC^r(M,U)$ is considered as an open subset of $\cC^r(M,E)$. 
\end{lem}

\begin{prf} Let $\pi \: TM \to M$ denote the bundle projection. 
Then both components of the map 
\begin{eqnarray*}
 \cC^r(M,U) &\to& \cC^{r-1}(TM,TU) \cong \cC^{r-1}(TM,U) \times \cC^{r-1}(TM,E), \\
f &\mapsto & Tf = (f \circ \pi, \dd f) 
\end{eqnarray*}
are restrictions of continuous linear maps, hence smooth. 
Since $f^*\alpha = \alpha \circ Tf$, 
smoothness of $\alpha$ and Lemma~\ref{lem:smooth-compo}
imply that
$$ \alpha_* \: \cC^{r-1}(TM,TU) \to \cC^{r-1}(TM, F), 
\quad h \mapsto \alpha \circ h 
$$
is smooth, from which the assertion follows  (recall that we topologize 
$\Omega^1_{\cC^{r-1}}(M,F)$ as a closed subspace of $\cC^{r-1}(TM,F)$).  
\end{prf}

\begin{prop} \mlabel{prop:smooth-logder0} For any Lie group $K$ with Lie algebra 
$\fk$, any compact manifold $M$ 
(possibly with boundary) and any $r \in \N \cup \{\infty\}$, the left logarithmic derivative 
$$ \delta \: \cC^r(M,K) \to  \Omega^1_{\cC^{r-1}}(M,\fk) $$ 
is a smooth map with respect to the Lie group structure on $\cC^r(M,K)$, 
and 
\begin{eqnarray}
  \label{eq:diff-delta}
T_\1(\delta) = \dd \: \cC^r(M,\fk) \to \Omega^1_{\cC^{r-1}}(M,\fk), \quad 
\xi \mapsto \dd \xi. 
\end{eqnarray}
\end{prop}

\begin{prf} From Lemma~\ref{lem:smoothprop}, we already know that the action 
of the Lie group $G := \cC^r(M,K)$ on $\Omega^1_{\cC^r}(M,\fk)$ by $f\cdot\alpha := \Ad(f)\alpha$ 
is smooth. Since the inclusion map 
$\cC^r(M,K) \to \cC^{r-1}(M,K)$ is a smooth morphism of Lie groups, the action of $G$ on 
on $\Omega^1_{\cC^{r-1}}(M,\fk)$ is also smooth. 

The product rule 
$\delta(\eta\gamma) 
= \delta(\gamma) + \Ad(\gamma)^{-1}\delta(\eta)$
means that $\delta$ is a right crossed homomorphism for the smooth action of 
$\cC^r(M,K)$ on $\Omega^1_{\cC^{r-1}}(M,\fk)$. 
It therefore suffices to verify its smoothness in a neighborhood 
of the identity. 

Let $(\phi_K, U_K)$ be a chart of an identity neighborhood of $K$ with 
$T_\1(\phi_K) = \id_\fk$, so 
that $(\phi_G,U_G)$ with 
$$ U_G := \lfloor M,U_K \rfloor := \{ \gamma \in G \: \gamma(M) \subeq U_K \}, 
\quad \phi_G(\gamma) := \phi_K \circ \gamma $$
is a chart of an identity neighborhood of the Lie group $G$. 
If $\kappa_K \in \Omega^1(K,\fk)$ denotes the left Maurer--Cartan form of $K$, then we have a map 
\begin{align*}
\cC^r(M, \phi_K(U_K)) &\to \Omega^1_{\cC^{r-1}}(M,\fk), \\ 
\xi &\mapsto \delta(\phi_K^{-1} \circ \xi) 
= (\phi_K^{-1} \circ \xi)^*\kappa_K 
= \xi^*(\phi_K^{-1})^*\kappa_K 
\end{align*}
whose smoothness follows from Lemma~\ref{lem:smoothprop2}. 

Set $\beta := (\phi_K^{-1})^*\kappa_K$. Then $\beta_0 = \id_\fk$
and (\ref{eq:diff-delta}) follows from the fact that 
for every $m \in M$ and $v \in T_m(M)$, we have 
\begin{align*}
\derat0 (t\xi)^*\beta v  
&= \derat0 \beta_{t\xi(m)}(t \dd\xi(m)v)
= \lim_{t \to 0} \beta_{t\xi(m)}\dd\xi(m)v
= \beta_0 \dd\xi(m)v 
= \dd\xi(m)v. 
\qedhere\end{align*}
\end{prf}

\section{$\cC^k$-regularity is an extension property} 
\mlabel{app:b}

In this appendix we generalize the result that regularity of 
Lie groups is an extension property to the stronger notion of
$\cC^k$-regularity for $k \in\N_0$. 

Throughout this section $I:=[0,1]$.
If $\gamma:(-\eps,\eps)\to G$ is a $\cC^1$ curve, then the left 
logarithmic derivative $\delta(\gamma):(-\eps,\eps)\to\g$ 
is defined by \[
\delta(\gamma)_t:=\dd\ell_{\gamma(t)^{-1}}(\gamma(t))\gamma'(t)
\] where $\ell_g:G\to G$ denotes the left translation $\ell_g(x):=gx$.

\begin{defn} \mlabel{def:regular0} 
Let $k \in \N_0 \cup \{\infty\}$. 
A Lie group $G$ with Lie algebra $\g$ is called {\it ${\cC^k}$-regular}, if for each
$\xi \in {\cC^k}(I,\g)$, the initial value problem 
\begin{eqnarray}
  \label{eq:lie-ivp}
\gamma(0) = \1, \quad \delta(\gamma) = \xi 
\end{eqnarray}
has a solution $\gamma_\xi$, which is then contained in 
$\cC^{k+1}(I,G)$, and the corresponding evolution map 
$$ \evol_G \: {\cC^k}(I,\g) \to G, \quad \xi \mapsto \gamma_\xi(1) $$
is smooth. The solutions of 
(\ref{eq:lie-ivp})  are unique whenever they exist (cf.\ \cite{Ne06}). 
If $G$ is ${\cC^k}$-regular, we write 
$$ \Evol_G \: {\cC^k}(I,\g) \to \cC^{k+1}(I,G), \quad \xi \mapsto \gamma_\xi $$
for the corresponding map on the level of Lie group-valued curves. 
This map is also smooth (cf.\ \cite[Thm. A]{Gl12}). 
The group $G$ is called {\it regular} if it is $C^\infty$-regular. 
\end{defn}

\begin{rem} \mlabel{rem:reg1a} (a) Any regular Lie group $G$ 
has a smooth  exponential function 
$$ \exp_G \: \g\to G \quad \hbox{ by } \quad \exp_G(x) := \gamma_x(1), $$
where $x \in \g$ is considered as a constant function $I \to \g$. 
As a restriction of the smooth function $\evol_G$ to the topological subspace 
$\g \subeq C^\infty(I,\g)$ of constant functions, 
the exponential function is smooth. 

(b) For $k \leq r$, the ${\cC^k}$-regularity of a Lie group $G$ implies its 
$\cC^r$-regularity because the inclusion map 
$\cC^r(I,\g) \to {\cC^k}(I,\g)$ is continuous linear, hence smooth. 
\end{rem}

\begin{lem}\mlabel{lem.e.2} 
Let $G$ be a Lie group with Lie algebra~$\g$. 
Then the prescription 
$$ \xi * \gamma := \delta(\gamma) + \Ad(\gamma)^{-1}\xi$$
defines a smooth affine right action of the group $\cC^{k+1}(I,G)$ on 
${\cC^k}(I,\g)$. 
\end{lem}

\begin{prf} That we have an action follows from 
\[ (\xi * \gamma_1) * \gamma_2 
= \delta(\gamma_2) + \Ad(\gamma_2)^{-1}(\delta(\gamma_1) + \Ad(\gamma_1)^{-1}\xi) \\
= \delta(\gamma_1\gamma_2) + \Ad(\gamma_1 \gamma_2)^{-1}\xi = \xi * (\gamma_1 \gamma_2). 
\] 

Since $I$ is a compact manifold with boundary, 
the smoothness of the action follows from the smoothness of 
$\delta$ (see Proposition~\ref{prop:smooth-logder0}) and Lemma~\ref{lem:smoothprop} (note that we can identify $\Omega_{\cC^k}(I,\g)$ with
$\cC^k(I,\g)$). 
\end{prf} 

\begin{rem} \mlabel{rem:7.5.7} If $\xi = \delta(\eta)$ for some smooth function 
$h \: M \to G$, then the Product Rule implies that 
$\delta(\eta\gamma) = \xi * \gamma,$ 
so that the action from above corresponds to the right multiplication action 
on the level of group-valued functions. 
\end{rem}

\begin{lem} \mlabel{lem:reg-crit} 
{\rm(Local regularity criterion)} 
Let $G$ be a Lie group with 
Lie algebra $\g$ and $k \in \N_0 \cup \{\infty\}$. Suppose that 
{\rm(\ref{eq:lie-ivp})} has a solution for each $\xi$ in an 
open $0$-neighborhood $U \subeq {\cC^k}(I,\g)$. 
Then it has a solution for each $\xi \in {\cC^k}(I,\g)$. 
If the evolution map 
$$ \evol_G \: {\cC^k}(I,\g) \to G, \quad \xi \mapsto \gamma_\xi(1) $$
is smooth in $U$, then it is smooth on all of ${\cC^k}(I,\g)$. 
\end{lem}

\begin{prf} Let $\xi \in \cC^k(I,\g)$. 
For $n \in \N$ and $i \in \{0,\ldots, n-1\}$ we define 
$$\xi_i^n \in \cC^k(I,\g) \quad \mbox{ by } \quad 
\xi_i^n(t) := \frac{1}{n}\xi\Big(\frac{i+t}{n}\Big). $$
Then for each $m \leq k$ we have  
$
\frac{d^m}{dt^m}\xi_i^n(t) = \frac{1}{n^{m+1}}\xi^{(m)}\Big(\frac{i+t}{n}\Big)
$.
We conclude that, for $n \to \infty$, the sequence 
$\frac{d^m}{dt^m}\xi_i^n$ tends to $0$ in $\cC^0(I,\g)$, uniformly in $i$, and hence that 
$\xi_i^n$ tends to $0$ in $\cC^k(I,\g)$, uniformly in $i$. 
In particular, there exists some $N \in \N$ for which 
$\xi_i^N \in U$ for $i = 0,1,\ldots, N-1$. 

We define a path $\gamma_\xi \: I \to G$ by 
$$ \gamma_\xi(t) := 
\gamma_{\xi_0^N}(1)\cdots \gamma_{\xi_{i-1}^N}(1)
\gamma_{\xi_i^N}(nt-i) 
\quad \mbox{ for } \quad \frac{i}{n} \leq t \leq \frac{i+1}{n} $$
and observe that $\delta(\gamma_\xi)$ exists on all of $I$ and 
equals $\xi$. We now put 
$$\Evol_G(\xi) := \gamma_\xi \quad \mbox{ and } \quad 
\evol_G(\xi) := \gamma_\xi(1). $$

Since, for each $i$, 
the assignment $\xi \mapsto \xi_i^N$ is linear and continuous, 
there exists an open neighborhood $V$ of $\xi$ such that 
$\eta_i^N \in U$ holds for each $\eta \in V$ and $i =0,1,\ldots, N-1$. 
It suffices to show that $\evol_G$ is smooth on $V$, 
but this follows from the fact that 
\[ \evol_G(\eta) = \evol_G(\eta_0^N) \cdots \evol_G(\eta_{N-1}^N)   \] 
is a product of $N$ smooth functions. 
\end{prf}

For the definition of an initial Lie subgroup see \cite[Def. II.6.1]{Ne06}.

\begin{prop} \mlabel{prop:reg-crit1} Let 
$G$ be a ${\cC^k}$-regular Lie group with Lie algebra $\g$ and 
$H \leq G$ an initial Lie subgroup with Lie algebra $\h \subeq \g$, for which 
there exists an open 
identity neighborhood $U \subeq G$ and a smooth function 
$f \: U \to E$ into some locally convex space $E$, such that 
$f$ is constant on $U \cap gH$ for each $g \in U$, and $H \cap U = f^{-1}(0) \cap U$. Then 
$H$ is $\cC^k$-regular.   
\end{prop}

\begin{prf} 
The ${\cC^k}$-regularity of $G$ implies the existence of a smooth evolution map 
$$ \evol_H \: {\cC^k}(I,\fh) \to G, \quad \xi \mapsto 
\gamma_\xi(1), $$
and since $H$ is initial, it suffices to see that the range 
of this map lies in $H$. 

If $\xi \in {\cC^k}(I,\fh)$ such that 
$\im(\gamma_\xi) \subeq U$, then  for every $t\in I$, 
$$ (f \circ \gamma_\xi)'(t) 
= \dd f({\gamma_\xi(t)})\gamma_\xi'(t) = 0 $$
because $\gamma_\xi'(t) = \dd\ell_{\gamma_\xi(t)}(\1)\xi(t)$ is 
the derivative of a curve in the set $U \cap \gamma_\xi(t)H$, on which 
$f$ is constant. 
Therefore $f \circ \gamma_\xi$ is constant, which leads to 
$\im(\gamma_\xi) \subeq f^{-1}(f(\1)) = f^{-1}(0) = U \cap H$. 

If $\xi \in {\cC^k}(I,\fh)$ is arbitrary, we apply the 
preceding argument to the curves 
$t\mapsto \gamma_\xi(t_0)^{-1} \gamma_\xi(t)$
on sufficiently small intervals $[t_0, t_0 + \eps]$ and see that 
$\im(\gamma_\xi)$ is contained in~$H$. 
\end{prf}

\begin{thm} \mlabel{thm:reg-ext-prop} 
{\rm($\cC^k$-regularity is an extension property)} Let $q \: \hat G \to G$ 
be an extension of the Lie group $G$ by the Lie group $N$ and $k \in \N_0$. Then 
$\hat G$ is ${\cC^k}$-regular if and only if $N$ and $G$ are ${\cC^k}$-regular. 
\end{thm}

\begin{prf} 
\textbf{Step 1.}
We assume that $G$ and $N$ are $\cC^k$-regular and show that this implies 
the $\cC^k$-regularity of $\hat G$. 
Since $G$ is $\cC^k$-regular, the evolution map 
$$\Evol_G \: {\cC^k}(I,\g) \to \cC^{k+1}_*(I,G)$$ 
is smooth \cite[Thm A]{Gl12}. 

Let $U_G \subeq G$ be an open $\1$-neighborhood for which we have a smooth 
section $\sigma \: U_G \to \hat G$ with $\sigma(\1_G) = \1_{\hat G}$ and 
$$ p \: {\cC^k}(I,\hat\g) \to {\cC^k}(I,\g), \quad \xi \mapsto q_\g \circ \xi $$
be the projection map. Then 
$ V := p^{-1}(\Evol_G^{-1}(\cC^{k+1}(I,U_G))) $
is an open $0$-neighborhood in ${\cC^k}(I,\hat \g)$. 
Further, by Lemma~\ref{lem:a.2}
the map 
$$ \Phi \: V \to \cC^{k+1}_*(I,\hat G), \quad \xi \mapsto \sigma \circ \Evol_G(q_\g \circ \xi) $$
is smooth. For $\xi \in {\cC^k}(I,\hat\g)$, we find 
\begin{align*}
&\quad q_\g(\xi * \Phi(\xi)^{-1}) 
= q_\g\Big(\Ad(\Phi(\xi))(\xi - \delta(\Phi(\xi)))\Big) \\
&= \Ad(q_G(\Phi(\xi)))\Big( q_\g \circ \xi - \delta(q_G(\Phi(\xi)))\Big) 
= \Ad(q_G(\Phi(\xi)))\big( q_\g \circ \xi - q_\g \circ \xi\big) = 0.
\end{align*}
This means that $\xi * \Phi(\xi)^{-1} \in {\cC^k}(I,\fn)$. 
Now Lemma~\ref{lem.e.2}, applied to the action of $\cC^{k+1}(I,\hat G)$ 
on $\Omega^1_{{\cC^k}}(I,\hat \g) \cong {\cC^k}(I,\hat\g)$, 
 shows that it depends smoothly on $\xi$. 
We thus obtain a smooth map 
$$ \widehat E \: V \to \hat G, \quad \widehat E(\xi) := 
\evol_N(\xi * \Phi(\xi)^{-1}) \cdot \Phi(\xi)(1). $$
The curve $\gamma := \Evol_N(\xi * \Phi(\xi)^{-1})\Phi(\xi)$ 
satisfies 
$$ \delta(\gamma) 
= \delta(\Phi(\xi)) + \Ad(\Phi(\xi))^{-1} 
\big(\delta(\Phi(\xi)^{-1}) + \Ad(\Phi(\xi))\xi\big) 
= \xi. $$
Therefore $\widehat E$ is a smooth local evolution map for $\hat G$, 
and Lemma~\ref{lem:reg-crit} implies that $\hat G$ is 
$\cC^k$-regular.

\textbf{Step 2.} Conversely, we show that 
the ${\cC^k}$-regularity of $\hat G$ implies the ${\cC^k}$-regularity of $N$ and~$G$. 
To see that $N$ is ${\cC^k}$-regular, we choose a chart 
$(\phi_G,U_G)$ of $G$ and consider the map 
$$ f := \phi_G \circ q \: q^{-1}(U_G) \to \phi_G(U_G), $$
which is constant on the left cosets of $N$, lying in this set. 
Therefore Proposition~\ref{prop:reg-crit1} implies that $N$ is ${\cC^k}$-regular because 
it is a submanifold of $\hat G$, hence in particular an initial submanifold. 

To see that $G$ is ${\cC^k}$-regular, we first choose a continuous linear section 
$\sigma \: \g \to \hat \g$, which induces a continuous linear section 
$$ \sigma_* \: {\cC^k}(I,\g) \to {\cC^k}(I,\hat\g), \quad 
\xi \mapsto \sigma \circ \xi. $$
Then, for each $\xi \in {\cC^k}([0,1],\g)$, the curve 
$\gamma_\xi := q \circ \Evol_{\hat G}(\sigma\circ \xi)$ satisfies 
$$ \delta(\gamma_\xi) = \L(q) \circ \sigma \circ \xi = \xi, $$
so that 
$\evol_G = q \circ \evol_{\hat G} \circ \sigma$ 
is a composition of smooth maps, hence smooth. 
\end{prf}

\end{document}